\documentclass{amsart}[12pt]

\usepackage{amsfonts,amssymb,amsmath,amscd,amsthm, mathrsfs}
\usepackage{graphicx}
\usepackage{epstopdf}
\usepackage{hyperref}
\usepackage{upgreek}

\newtheorem{thm}{Theorem}[section]
\newtheorem{dfn}{Definition}[section]
\newtheorem{prop}{Proposition}[section]
\newtheorem{lm}{Lemma}[section]

\newtheorem{rem}{Remark}[section]
\newtheorem*{con}{Conjecture}
\newtheorem*{main}{Main Theorem}
\newtheorem{ex}{Example}[section]

\newcommand{\mb}[1]{\mathbb{#1}}
\newcommand{\mf}[1]{\mathfrak{#1}}
\newcommand{\mc}[1]{\mathcal{#1}}
\newcommand{\rr}{\mathbb{R}}
\newcommand{\zz}{\mathbb{Z}}
\newcommand{\cc}{\mathbb{C}}

\renewcommand{\Im}{\operatorname{Im}}

\newcommand{\al}{\alpha}

\newcommand{\dl}{\delta}

\newcommand{\ep}{\varepsilon}

\newcommand{\Lmd}{\Lambda}
\newcommand{\lmd}{\lambda}
\newcommand{\om}{\omega}
\newcommand{\Om}{\Omega}
\newcommand{\sg}{\sigma}

\newcommand{\tht}{\theta}
\newcommand{\kp}{\kappa}

\newcommand{\ey}{\frac{1}{2}}

\newcommand{\ac}{\mathcal{A}}
\newcommand{\ack}{\mathcal{A}_{\I}}
\newcommand{\kk}{K_{\I}}
\newcommand{\uk}{U_{\I}}
\newcommand{\lk}{L_{\I}}

\newcommand{\N}{\mathbf{N}}
\newcommand{\nf}{\mf{n}}
\newcommand{\I}{\mathbf{I}}

\newcommand{\se}{\frac{2}{3}}

\newcommand{\lmh}{\hat{\Lmd}}
\newcommand{\ldn}{\Lmd^{D_N}_N}
\newcommand{\hld}{\hat{\Lmd}^{D_N}_N}

\newcommand{\lo}{\Lmd^+_{\om}}
\newcommand{\hlo}{\hat{\Lmd}^+_{\om}}

\newcommand{\zd}{\dot{z}}
\newcommand{\zb}{\bar{z}}
\newcommand{\ze}{z^{\ep}}
\newcommand{\zey}{z^{\ep_1}}
\newcommand{\zel}{z^{\ep_0}}

\newcommand{\zo}{z^{\om}}
\newcommand{\ztl}{\tilde{z}}
\newcommand{\zh}{\hat{z}}

\newcommand{\xh}{\hat{x}}
\newcommand{\xd}{\dot{x}}

\newcommand{\vt}{\tilde{v}}
\newcommand{\vte}{\tilde{v}^{\ep}}
\newcommand{\sgt}{\tilde{\sg}}

\newcommand{\lmn}{\lmd_n}

\begin{document}

	\title{Simple choreographies of the planar Newtonian $N$-body Problem}
	\author{Guowei Yu}
	\email{yu@math.utoronto.ca}
	\date{08/29/2016}
	
	\address{Department of Mathematics, University of Toronto}
	
	\begin{abstract} 
	      In the $N$-body problem, a simple choreography is a periodic solution, where all masses chase each other on a single loop. In this paper we prove that for the planar Newtonian $N$-body problem with equal masses, $N \ge 3$, there are at least $2^{N-3} + 2^{[(N-3)/2]}$ different main simple choreographies. This confirms a conjecture given by Chenciner and etc. in \cite{CGMS02}. 
	\end{abstract}
	
	\maketitle
	
\section{Introduction} \label{sec intro}

The Newtonian $N$-body problem describes the motion of $N$ point masses under the attraction of each other according to Newton's gravitational law. When all the masses are equal, there exist periodic solutions, where all the masses travel on a single loop (it is still an open problem whether such a solution exists when the masses are unequal \cite{C04}; throughout the paper we assume all the masses are equal and $N \ge 3$). Such solutions usually satisfy certain symmetric constraints, as if they are dancing according to certain choreographies. This inspired Carles Sim\'o to name them \emph{simple choreographies}. 

A trivial example of simple choreographies is the \emph{rotating $N$-gon}, where the $N$ point masses form a regular $N$-gon at each moment and rotate rigidly around the center of mass at a constant angular velocity. The first non-trivial simple choreography is the now famous \emph{Figure-Eight} solution, which was discovered numerically by Moore \cite{Mr93}, and then independently and rigorously proved by Chenciner and Montgomery \cite{CM00}. This remarkable solution immediately caught a lot of attention and many simple choreographies were found numerically: the \emph{Super-Eight} solution of the $4$-body problem by Gerver and several families of simple choreographies for different values of $N$ by Sim\'o (\cite{CGMS02}, \cite{Si00}). Some more recently numerical discoveries can be found in \cite{MS13}. 

On the other hand to rigorously prove their existence is a much harder task. Following \cite{CM00}, the idea is to find a simple choreography as a minimizer of the Lagrangian action functional among loops satisfying certain constraints. For Newtonian potential, as already noticed by Poincar\'e, when two or more bodies collide, the action functional may still be finite. As a result, the desired minimizer may contain collision, which prevents it to be a real solution. This is the main obstacle to apply variational methods to the Newtonian $N$-body problem. 

A lot of progresses have been made to overcome this difficulty since the proof of the \emph{Figure-Eight} solution. We briefly summarize them as following. 

\begin{itemize}	
	\item \textbf{Local deformation}: assuming there is an isolated collision along the minimizing path, one tries to show that after a small deformation near the isolated collision, one gets a new path with strictly smaller action value, which gives a contradiction. For the details see \cite{Mo98}, \cite{Ve02}, \cite{Y16b}, \cite{Mc02}, \cite{C02} and \cite{FT04}. In the last three references the existence of such a local deformation was implied implicitly through Marchal's average method. 
			
	\item \textbf{Level estimate}: one gives a sharp lower bound estimate of the action functional among all the collision paths in the admissible class and then try to find a test path within the admissible class, such that its action value is strictly smaller than the previous lower bound estimate. The Figure-Eight solution was originally proved in \cite{CM00} using this method. Results obtained using this method can also be found in \cite{Ch01}, \cite{Ch03}, \cite{Ch08}, \cite{CL09}, \cite{COX12}, \cite{Ch13} and the references within. 
\end{itemize}

Despite of the above progresses, up to our knowledge very few simple choreographies have been rigorously proved. Besides the Figure-Eight, the Super-Eight was first proved in \cite{KZ03} with rigorous numerical method and then analytically in \cite{Sh14}. The best result so far is obtained by Ferrario and Terracini in \cite{FT04}, where they proved there is at least one Fight-Eight type simple choreography for every odd $N.$

One way to bypass the difficulty is to change the potential from Newtonian to \emph{strong force} (for a precise definition see the end of Section \ref{sec proof}). For a strong force potential, the action value of a path with any collision must be infinity. With this result, in \cite{CGMS02} the authors showed that there are infinitely many distinct \textbf{main} simple choreographies, when the potential is a strong force (following \cite{CGMS02}, by a \textbf{main} simple choreography, we mean it is not derived from a given simple choreography either by traveling around it multiple times or by a continuation of the angular momentum, or a combination of both). For the Newtonian potential, based on numerical discoveries, they made the following conjecture in the same paper. 

\begin{con} \label{conjecture}    
      For every $N \ge 3$, there is a \textbf{main} simple choreography solution for the equal mass Newtonian $N$-body problem different from the trivial circle one (i.e., the rotating $N$-gon). The number of such distinct \textbf{main} simple choreographies grows rapidly with $N$.  
\end{con}

Let $[\cdot]$ denote the integer part of a real number. We confirm this conjecture by proving the following result in this paper. 

\begin{main} 
 For every $N \ge 3$, there are at least $2^{N-3} + 2^{[(N-3)/2]}$ different \textbf{main} simple choreographies for the equal mass Newtonian $N$-body problem. 
\end{main}

We remark that if the condition of being \textbf{main} simple choreographies was dropped, then the existence of infinitely many simple choreographies have been established by Chenciner and F{\'e}joz in \cite{CF09}.

In the rest of the paper, we will only consider the planar $N$-body problem. For simplicity, we also assume every mass is one: $m_j = 1,$ for any $j \in \N := \{0, 1, \dots, N-1\}.$ Let $z = (z_j)_{j \in \N} \in \cc^N$ be the positions of the masses, then it satisfies the following equations
\begin{equation}
 \label{N body} \ddot{z}_j = \sum_{k \in \N \setminus \{j\}} - \frac{z_j - z_k}{|z_j - z_k|^3}, \quad j \in \N.
\end{equation}
This is the Euler-Lagrange equation of the action functional
$$ \ac(z, T_1, T_2) := \int_{T_1}^{T_2} L(z, \zd) \,dt, \quad \quad L(z, \zd) := K(\zd) + U(z), $$
where $K(\zd)$ is the kinetic energy and $U(x)$ is the (negative) potential energy: 
$$ K(\zd) := \ey \sum_{j \in \N} |\zd_j|^2, \quad \quad U(z) := \sum_{\{j < k \} \subset \N} \frac{1}{|z_j - z_k|}. $$
Furthermore we set $\ac(z, T) := \ac(z, 0, T).$

By the homogeneity of the potential, given a periodic solution, one can find another one with any prescribed period. Because of this, we will only consider periodic solutions with period $N$.

Following \cite{CM00} and \cite{CGMS02}, the key idea is to impose proper symmetric constraints on the loop space, then obtain the simple choreographies as collision-free minimizers. We recall it briefly in the following. 

Let $\Lmd_N = H^1(\rr/N\zz, \cc^{N})$ be the space of Sobolev loops and $\hat{\cc}^N := \{ z \in \cc^{N}: z_j \ne z_k, \; \forall \{j \ne k \} \subset \N \}$ the space of collision-free configurations, then $\lmh_N= H^1(\rr/N\zz, \hat{\cc}^N)$ is the subspace of collision-free loops. Given a finite group $G$, we can define its action on $\Lmd_N$ as following:
\begin{equation}
 \label{symmetry} 
 g\big( z(t) \big)= \big( \rho(g) z_{\sg(g^{-1})(0)} (\uptau(g^{-1})t), \dots, \rho(g) z_{\sg(g^{-1})(N-1)} (\uptau(g^{-1})t) \big), \quad \forall g \in G, 
\end{equation}
where 
\begin{enumerate}
 \item $\uptau: G \to O(2)$ representing the action of $G$ on the time circle $\rr/N\zz,$
 \item $\rho: G \to O(2)$ representing the action of $G$ on the $2$-dim Euclid space,
 \item $\sg: G \to \mc{S}_N$ representing the action of $G$ on the index set $\N. $
\end{enumerate}

$\Lmd^G_N := \{ z \in \Lmd_N: g(z(t)) = z(t), \; \forall g \in G \}$ is the space of $G$-equivariant loops. As all masses are equal, the action functional $\ac$ is invariant under the above group action in $\Lmd^G_N$. By the symmetric critical principle of Palais \cite{Pa79}, a critical point of $\ac$ in $\hat{\Lmd}^G_N$ ($\hat{\Lmd}^G_N := \Lmd^G_N \cap \lmh_N$)  is also a critical point of it in $ \hat{\Lmd}_N$, and therefore a solution of equation \eqref{N body}.

Consider the cyclic group $\zz_N = \langle g | \; g^N = 1 \rangle $ with the following action
\begin{equation}
 \label{g} \uptau(g) t = t-1, \quad \rho(g) = \text{identity}, \quad \sg(g) = (0,1, \dots, N-1).
\end{equation}
Then the following holds for any $z =(z_j)_{j \in \N} \in \Lmd^{\zz_N}_N$, 
\begin{equation}
 \label{chore eq} z_j(t) = z_0(j+t),\;\; \forall t \in \rr, \;\; \forall j \in \N. 
\end{equation}
This means any collision-free critical point of $\ac$ in $\Lmd^{\zz_N}_N$ must be a simple choreography. The most obvious critical point is of course the global minimizer. However as proved in \cite{BT04}, the global minimizer in $\Lmd^{\zz_N}_N$ is nothing but the \emph{rotating $N$-gon}. 

A possible way to get non-trivial simple choreographies is to consider a group $G$ with proper action, such that it contains $\zz_N$ with action given in \eqref{g} as a subgroup, and meanwhile the symmetric constraints is strong enough to rule out the rotating $N$-gon. This was exactly the idea used in \cite{CM00} and \cite{FT04}. 

However the real gold deposits locate on the topological constraints. As we can see the space of collision-free loops $\hat{\Lmd}^{\zz_N}_N$ has infinitely many connected components. A nice way of distinguishing these components is through the braids group, for the details see the beautiful paper by Montgomery \cite{Mo98}. In principle it is possible to find a main choreography solution in each connected component. For a strong force potential, this was indeed proved in \cite{CGMS02}.

Meanwhile for the Newtonian potential, it is much more difficult to show that a minimizer is collision-free when topological constraints are involved. First \emph{local deformation} usually fails in this case, because to be able to lower the action the possible directions of local deformation are restricted and in many cases you end up within a topological class different from the required one. Second although in principle \emph{level estimate} should work under topological constraints, in practice except some special cases, it is very difficult to give an accurate lower bound estimate of the action values of the collision paths.

To overcome the difficulty caused by the topological constraints, we propose to introduce an additional \emph{monotone constraints} (see Definition \ref{mc}). This constraints provide further information regarding the relative positions of the masses. With such information, first we can rule out collisions involving more than two masses, and second when there is a binary collision, it allows us to make certain \emph{global deformation} of the path to get a new one with strictly lower action value and reach a contradiction. We believe this is the first time such idea has been used in this classic problem. The details will be given in Section \ref{sec proof}.

In the rest of this section, a proof of the \textbf{Main Theorem} will be given. Let's consider the dihedral group
$D_N = \langle g, h |\; g^N = h^2= 1, \; (gh)^2 =1 \rangle$ with the action of $g$ defined as in \eqref{g} and $h$ as following
\begin{equation}
 \label{h} \uptau(h) t = -t+1, \quad \rho(h) q = \bar{q}, \quad \sg(h) = (0, N-1)(1, N-2)\cdots(\nf, N-1-\nf),
\end{equation}
where $q \in \cc$, $\bar{q}$ is its conjugate, and 
\begin{equation}
\label{eq: N star}  \nf := [(N-1)/2].
\end{equation}

By the above definition, $\ldn \subset \Lmd^{\zz_N}_N$. However the rotating $N$-gon is also contained in $\ldn$. Hence the minimizer of $\ac$ in $\ldn$ is again such a trivial solution. In order to get non-trivial simple choreographies, topological constraints need to be added into the problem. 

First by the symmetric constraints, $z =(z_j)_{j \in \N} \in \ldn$ must satisfy \eqref{chore eq} and 
\begin{equation}
\label{eq: symm 1} z_j(t) = \bar{z}_{N-1-j}(1-t), \;\; \forall t \in \rr, \; \forall j \in \N.
\end{equation}
In particular, 
\begin{equation}
\label{eq: symm 3} z_j(0)  = \bar{z}_{N-j}(0), \; \text{ if } j \in \N \setminus \{0\}; \;\; z_j(1/2)  = \bar{z}_{N-1-j}(1/2), \; \text{ if } j \in \N.
\end{equation}
The symmetric constraints also imply 
\begin{equation}
\label{eq: symm 2} z_0(t) = \bar{z}_0(-t), \;\; \forall t \in \rr,
\end{equation}
\begin{equation}
\label{eq: z0 zj zN-j} z_0(j+t) = \begin{cases}
z_j(t), \;\; & \forall t \in [0, 1/2], \\
\zb_{N-1-j}(1-t), \;\; & \forall t \in [1/2, 1],
\end{cases} \;\; \text{if } j \in \{0, 1, \dots, \nf\}.
\end{equation}
As a result, if $z \in \ldn$ is collision-free, it must satisfies the following:
\begin{equation}
\label{eq: z0 top const} \Im(z_0(j/2)) \ne 0, \;\; \forall j \in \N \setminus \{0\}.
\end{equation}

Based on the above observation, we set 
\begin{equation}
\label{eq: Om_N} \Om_N := \{ \om = (\om_j)_{j \in \N \setminus \{0\}} : \om_j \in \{\pm 1\}, \; \forall j \in \N \setminus \{0\} \}.
\end{equation}
\begin{dfn}
For any $\om \in \Om_N$, we say a loop $z \in \ldn$ satisfies the \textbf{$\om$-topological constraints}, if 
\begin{equation}
 \label{eq: top con} \Im(z_0(j/2)) = \om_j |\Im(z_0(j/2))|, \;\; \forall j \in \N \setminus \{0\}.
 \end{equation}. 
\end{dfn}
Notice that if $z \in \hld$, then 
\begin{equation}
\label{eq: top con 2} \Im(z_0(j/2)) \ne 0 \text{ and } \frac{\Im(z_0(j/2))}{|\Im(z_0(j/2))|} =\om_j,\;\; \forall j \in \N \setminus \{0\}.
\end{equation} 

\begin{thm} \label{thm 1}
	For each $\om \in \Om_N$, there is at least one simple choreography $z^{\om}=(\zo_j)_{j \in \N} \in \hld$ satisfying \eqref{N body}, the $\om$-topological constraints and the following 
	\begin{equation}
	\label{eq: dot x > 0} \dot{x}_0^{\om}(t) > 0, \; \forall t \in (0, N/2), \text{ and } \dot{x}_0^{\om}(0) = \dot{x}_0^{\om}(N/2)=0, 
	\end{equation}
	where $x_0^{\om}(t) = \Im(\zo_0(t))$. 	
\end{thm}

Such $\zo$'s will be obtained as collision-free minimizer. The detailed proof of the above theorem will be given in Section \ref{sec proof}. For the moment we use it to give a proof of the \textbf{Main Theorem}. 

\begin{proof}
Although there are $2^{N-1}$ different $\om$'s in $\Om_N$. The number of distinct \text{main} simple choreographies is smaller. As some of the $\zo$'s obtained in Theorem \ref{thm 1} are identical to each other after a time reversing, a rigid motion or a combination of both. 

First for any $\om \in \Om_N$, $\zo(-t)$ satisfies the $(-\om)$-monotone constraints, so we only need to count the $\om$'s from the set $\Om^+_N: = \{ \om \in \Om_n: \; \om_1 = 1 \}$. 

Meanwhile let $\om^* = (\om^*_j)_{j \in \N \setminus \{0\}} \in \Om_N$ be defined as $\om^*_j = \om_{N-j}$. Then $-\zo(-t)$ satisfies the $\om^*$-monotone constraints. To exclude those, consider the following disjoint union $\Om_N^+= \Om_N^{+, -} \cup \Om_N^{+, +}$, where 
$$ \Om_N^{+, -}: = \{ \om \in \Om_N^+: \om_{N-1}=-1 \}, \; \; \Om_N^{+, +}:= \{ \om \in \Om_N^{+}: \om_{N-1} =1 \}. $$

There are $2^{N-3}$ different $\om$'s in $\Om_N^{+, +}$. However if $\om \in \Om_N^{+, +}$, so is $\om^* \in \Om_N^{+, +}$. Meanwhile notice that there are $2^{[(N-2)/2]}$ different $\om$'s in $\Om_N^{+,+}$ with $\om^*= \om$. As a result, the number of effective $\om$ we can count from $\Om_{N}^{+, +}$ is 
\begin{equation}
\label{eq: count om ++} 2^{N-4} +2^{[(N-4)/2]} = \ey \big( 2^{N-3}+ 2^{[(N-2)/2]} \big).
\end{equation}

Now for any $\om \in \Om^{+, -}_N$, although $\om^* \notin \Om_N^{+, -}$, but $-\om^*$ is. At the same time, notice that, there are $2^{[(N-2)/2]}$ different $\om \in \Om_N^{+,-}$ with $\om= -\om^*$, if $N$ is odd, and this number is zero, if $N$ is even (this is because, when $N$ is even, $\om=-\om^*$ implies $\om_{N/2}= -\om_{N/2}$, which can never happen, as $\om_{N/2} \in \{\pm1 \}$). As a result, the number of effective $\om$, we can count from $\Om_N^{+,-}$ is 
\begin{equation}
\label{eq: count om +-} \begin{cases}
 2^{N-4}, \; & \text{ if } N \text{ is even}, \\
 2^{N-4}+2^{[(N-4)/2]}, \; & \text{ if } N \text{ is odd}. 
\end{cases}
\end{equation}

By \eqref{eq: count om ++} and \eqref{eq: count om +-}, there are $2^{N-3}+2^{[(N-3)/2]}$ distinct main simple choreographies, and this finishes our proof.
\end{proof}

The number $2^{N-3} +2^{[(N-3)/2]}$ obtained above is the same as the number of simple choreographies belonging to a special family called \emph{linear chains}, see \cite[Proposition 5.1]{CGMS02}. This is not a coincidence. The family of linear chains consists of all simple choreographies which look like a chain of consecutive bubbles placed along the real axis (Both the Figure-Eight solution and the Super-Eight solution are from this family). As each simple choreography $\zo \hat{\Lmd}_{N}^{D_N}$ obtained in Theorem \ref{thm 1} satisfies \eqref{eq: symm 2}, \eqref{eq: top con 2} and \eqref{eq: dot x > 0}, it must belong to the family of linear chains as well. 

For each $\om$, the number of bubbles of the corresponding $\zo$ has a minimum determined by $\om$. For example, if $\om_j=1$, $\forall j \in \N \setminus \{0\}$, then the minimum is $1$. In fact, we know in this case $\zo$ must be the rotating $N$-gon, which consists exactly one bubble. However in general we don't know if the number of bubbles of $\zo$ is exactly this minimum. 

Our paper is organized as following: in Section \ref{sec proof}, we give a of Theorem \ref{thm 1}; in Section \ref{sec: addsym}, we consider simple choreographies with extra symmetries; in the last section, the Appendix, the proof of a technical lemma will be given. 

\textbf{Notations:} the following notation rules will apply through out the paper. 
\begin{itemize}
\item $i$ will always represent $\sqrt{-1}$;
\item Given a path $z(t)$ of the $N$-body, then $z_j(t)$ is the corresponding path of $m_j$, for any $j \in \N$ with $z_j(t)= x_j(t)+iy_j(t)$ and $x_j(t), y_j(t) \in \rr$;
\item If instead of $z(t)$, a path of the $N$-body is denoted by $\ztl(t), \zo(t), z^{\ep}(t) \dots$, then the corresponding changes will be made on $z_j(t), x_j(t)$ and $y_j(t)$;
\item Given any two non-negative integers $j_0, j_1$: 
$$ \{j_0, \dots, j_1 \} := \begin{cases}
\{j \in \zz: \; j_0 \le j \le j_1\}, \; & \text{ if } j_0 \le j_1, \\
\emptyset, \; & \text{ if } j_0 > j_1;
\end{cases} $$
\item $C_1, C_2, \dots$ representing positive constants vary from proof to proof. 
\end{itemize}

\section{Proof of Theorem \ref{thm 1}} \label{sec proof}
When there is no confusion, in this section $\ldn, \hld$ will simply be written as $\Lmd, \hat{\Lmd}$ respectively.

Due to the symmetric constraints, a loop $z \in \Lmd$ is entirely determined by $z(t), t \in [0, 1/2]$, so it is enough to focus on the \emph{fundamental domain}: $[0, 1/2].$ In the rest of the paper, when we try to define a loop $z \in \Lmd$, only $z(t), t \in [0, 1/2]$, will be given explicitly (the rest will follow from the symmetric constraints). Furthermore in this section, we set $\ac(z) := \ac_K (z) + \ac_U (z)$, where 
$$ \ac_K(z) := \int_0^{1/2} K(\zd) \,dt, \;\; \ac_U(z): = \int_0^{1/2} U(z) \,dt.$$

First let's introduce the monotone constraints. 
\begin{dfn} \label{mc}
	We say a loop $z \in \Lmd$ satisfies the \textbf{monotone constraints}, if 
	$$ x_0(j/2) \le x_0(j/2+t) \le x_0(j/2+1/2),\;\; \forall t \in [0, 1/2], \; \forall j \in \N;$$ 
	and the \textbf{strictly monotone constraints}, if 
	$$ x_0(t_1) < x_0(t_2), \;\; \text{ for any } 0 \le t_1 < t_2 \le N/2. $$ 
\end{dfn}

Let $\Lmd^+$ be the subset of all loops in $\Lmd$, which satisfy the monotone constraints and the following inequalities
\begin{equation}
\label{0} x_0(0) \le 0 \le x_0(N/2).
\end{equation}

By \eqref{eq: symm 1}, \eqref{eq: symm 3} and \eqref{eq: z0 zj zN-j}, $z \in \Lmd$ satisfying the monotone constraints is equivalent to the following:
\begin{equation} \label{eq: monotone 1} 
\begin{cases}
x_j(0) \le x_j(t) \le x_j(1/2), \; &\text{ if } j \in \{0, \dots, \mf{n}\},\\
x_j(0) \ge x_j(t) \ge x_j(1/2), \; &\text{ if } j \in \N \setminus \{0, \dots, \mf{n}\}, 
\end{cases} \;\; \forall t \in [0, 1/2];
\end{equation}
\begin{equation} 
\label{eq: monotone 2} \begin{split}
& x_0(0) \le x_0(1/2) = x_{N-1}(1/2) \le x_{N-1}(0)= x_1(0) \le  \cdots  \\
                    & \cdots \le x_{n+1}(0) = x_{n-1}(0) \le x_{n-1}(1/2)= x_n(1/2) \le x_n(0), \; \text{ if } N=2n;
\end{split}
\end{equation}
\begin{equation}
\label{eq: monotone 3} \begin{split}
& x_0(0) \le x_0(1/2) = x_{N-1}(1/2) \le x_{N-1}(0)= x_1(0) \le  \cdots  \\
& \cdots \le x_{n-1}(1/2)= x_{n+1}(1/2) \le x_{n+1}(0)= x_n(0) \le x_n(1/2), \; \text{ if } N =2n+1, 
\end{split}
\end{equation}

As a result, the motion of the masses satisfy certain monotone constraints along the horizontal direction. For illuminating pictures, see Figure \ref{fig: C0}. Roughly speaking the masses are located on the solid vertical lines, when $t =0$; on the dashed vertical lines, when $t =1/2$. Furthermore between $t=0$ and $t=1/2$, each mass is confined within the unique vertical strip bounded by the neighboring vertical lines. Whether the masses are above or below the real axis on each vertical line are determined by the $\om$-topological constraints. The pictures in Figure \ref{fig: C0} are corresponding to $\om \in \Om_N$ with $\om_j=1$, for any $i \in \N \setminus \{0\}$. 

\begin{figure}
	\centering
	\includegraphics[scale=0.65]{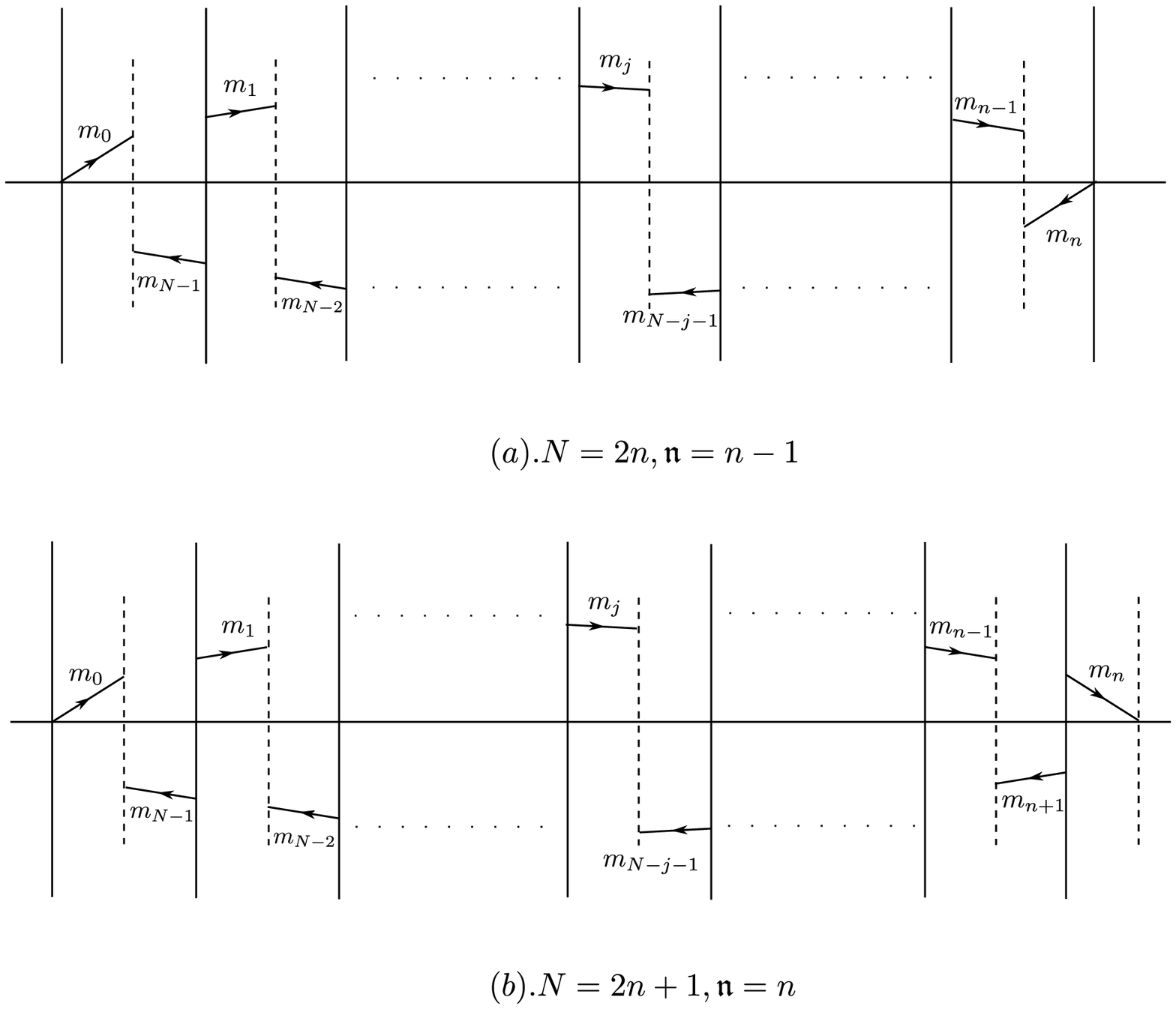}
	\caption{}
	\label{fig: C0}
\end{figure}

To put the topological constraints back into the problem. For any $\om \in \Om_N$, we define $\lo$ to be the subset of loops in $\Lmd^+$, which satisfy the $\om$-topological constraints. Correspondingly we set $ \hat{\Lmd}^+: = \Lmd^+ \cap \hat{\Lmd}$, $\hlo: = \lo \cap \hat{\Lmd}.$

\begin{prop}
	\label{lm 1} For each $\om \in \Om_N$, there is a $z^{\om} \in \lo,$ which is a minimizer of the action functional $\ac$ in $\lo.$
\end{prop}

\begin{proof}
	It is well known the action functional $\ac$ is weakly semi-lower continuous with respect to the Sobolev norm $H^1$. Meanwhile $\lo$ is weakly closed with respect to the same norm. Therefore we just need to show $\ac$ is coercive in $\lo$. 

	Choose a sequence of loops $\{z^k \in \lo\}$ with $\|z^k\|_{H^1}$ goes to infinity, as $k$ goes to infinity, it is enough to show that $\ac(z^k, N)$ goes to infinity as well.

	Recall that $z^k_0(t)= x^k_0(t)+iy^k_0(t)$. By \eqref{eq: symm 2}, $y^k_0(0)=0$, then 
	$$ |y_0^k(t_0)| = |y^k_0(t_0)-y^k_0(0)| \le \int_0^N |\dot{y}^k_0(t)| \,dt, \;\; \forall t_0 \in [0, N). $$
	By Cauchy-Schwartz inequality, 
	$$ |y^k_0(t_0)|^2 \le \big( \int_0^N |\dot{y}^k_0(t)| \,dt \big)^2 \le N \int_0^N |\dot{y}^k_0(t)|^2 \,dt, \;\; \forall t_0 \in [0, N). $$
	Hence 
	$$ \int_0^N |y^k_0(t)|^2 \,dt \le N^2 \int_0^N |\dot{y}^k_0(t)|^2 \,dt. $$

	Meanwhile by \eqref{0}, there is always a $t_k \in [0, N/2]$, such that $x^k_0(t_k)=0$. Then by similar computations as above, 
	$$ \int_0^N |x^k_0(t)|^2 \,dt \le N^2 \int_0^N |\dot{x}^k_0(t)|^2 \,dt. $$
	As a result,
	\begin{equation}
	\label{eq: zk le dot zk} \int_0^N |z^k_0(t)|^2 \,dt \le N^2 \int_0^N |\dot{z}^k_0(t)|^2 \,dt. 
	\end{equation} 
	Because $z^k \in \lo$, it satisfies \eqref{chore eq}. Then 
	\begin{equation}
	\label{eq: zk H1 norm}  \|z^k\|^2_{H^1} = \int_0^N \sum_{i \in \N} \big( |z^k_j(t)|^2 + |\dot{z}^k_j(t)|^2\big) \,dt = N \int_0^N |z^k_0(t)|^2+|\dot{z}^k_0(t)|^2 \,dt.
	\end{equation}
	Notice that the action functional satisfies 
	\begin{equation}
	\label{eq: A ge dot z} \ac(z^k, N) \ge \ey \int_0^N \sum_{j \in \N} |\zd^k_j(t)|^2 \,dt = \frac{N}{2} \int_0^N |\zd_0^k(t)|^2 \,dt.
	\end{equation}
	Together \eqref{eq: zk le dot zk}, \eqref{eq: zk H1 norm} and \eqref{eq: A ge dot z} imply
	$$ \ac(z^k, N) \ge \frac{1}{2(N^2+1)}\|z^k\|^2_{H^1}. $$
	As a result, $\ac(z^k, N)$ goes to infinity, as $k$ goes to infinity.  

\end{proof}

We need to prove the action minimizer $\zo$ is collision-free. However because of the monotone constraints, this does not necessarily mean it is a solution of equation \eqref{N body}. In addition, we also need to show it satisfies the \emph{strictly monotone constraints}. As this property will be useful in showing $\zo$ is collision-free, we will prove it first. 

\begin{lm}
	\label{nondecreasing} For any $\om \in \Om_N$, if $\zo \in \lo$ is an action minimizer of $\ac$ in $\lo$, then 
	 $$ x^{\om}_0(t_1) \le x^{\om}_0(t_2), \;\; \text{ for any } 0 \le t_1 < t_2 \le N/2.$$
\end{lm}
	
\begin{proof}
	For simplicity, let $z = \zo$. As $z \in \lo$, by the definition of monotone constraints, it is enough to show that, for any $0 \le t_1 \le t_2 \le 1/2$, 
	$$ x_0(j/2+t_1) \le x_0(j/2+t_2), \;\; \forall j \in \N. $$
	By a contradiction argument, let's assume there exist a $j_0 \in \N$ and $0 \le t_1 < t_2 \le 1/2$, such that
	$$ x_0(j_0/2 +s_1) >  x_0(j_0/2 +s_2) , \;\; \text{ for any } t_1 \le s_1 < s_2 \le t_2. $$
	By \eqref{eq: z0 zj zN-j}, depend on $j_0$ is even or odd, this is equivalent to 
	\begin{equation*}
	 \begin{cases}
	 x_{k}(s_1) > x_{k}(s_2), \; &\text{ if } j_0 =2k, \\
	 x_{N-k}(1/2 - s_1) > x_{N-k}(1/2-s_2), \; &\text{ if } j_0 =2k-1,
	 \end{cases} \; \text{ for any } t_1 \le s_1 < s_2 \le t_2.
	 \end{equation*} 
	Let $\ep= x_0(j_0/2 +s_1)-x_0(j_0/2 +s_2)>0$. We discuss it in two cases correspondingly. 

	\emph{Case 1}: $j_0$ is even ($j_0=2k$). We define a new loop $\ze \in \lo$ as following:
	$$ \ze_k(t) = \begin{cases}
	z_k(t)-2\ep, \;\; & \forall t \in [0, t_1], \\
	\big( 2x_k(t_2)-x_k(t) \big)+iy_k(t), \;\; & \forall t \in [t_1, t_2], \\
	z_k(t), \;\; & \forall t \in [t_2, 1/2], 
	\end{cases} $$
	\begin{equation*}
	\ze_j(t) = \begin{cases}
	 z_j(t), \;\; & \text{if } j \in \{k+1, \dots, N-1-k\}, \\
	 z_j(t)- 2\ep, \;\; & \text{if } j \in \N \setminus \{k, \dots, N-1-k\}, 
	\end{cases}
	\;\; \forall t \in [0, 1/2].
	\end{equation*}
	By the above definition of $\ze$,  $\ac_K(\ze) = \ac_K(z)$. Meanwhile by the monotone constraints, $\ac_U(\ze)<  \ac_U(z)$. Therefore $\ac(\ze) < \ac(z)$, which is absurd. 

	\emph{Case 2}: $j_0$ is odd ($j_0 =2k-1$). Similarly we define a new loop $\ze \in \lo$ as following: 
	$$ \ze_{N-k}(t) = \begin{cases}
	z_{N-k}(t), \;\; & \forall t \in [0, 1/2-t_2], \\
	\big( 2x_{N-k}(1/2-t_2)-x_{N-k}(t)\big)+iy_{N-k}(t), \;\; & \forall t \in [1/2-t_2, 1/2-t_1], \\
	z_{N-k}(t)-2\ep, \;\; & \forall t \in [1/2-t_1, 1/2], 
	\end{cases} $$
	\begin{equation*}
	\ze_j(t) = \begin{cases}
	 z_j(t), \;\; & \text{if } j \in \{k, \dots, N-1-k\}, \\
	 z_j(t)- 2\ep, \;\; & \text{if } j \in \N \setminus \{k, \dots, N-k\}, 
	\end{cases}
	\;\; \forall t \in [0, 1/2].
	\end{equation*}
	The rest is the same as in \emph{Case 1}. 
\end{proof}

Next we need to exclude the degenerate case, where the masses always stay on a single vertical line.
\begin{lm}
	\label{degenerate} For any $\om \in \Om_N$, if $\zo \in \lo$ is an action minimizer of $\ac$ in $\lo$, then $x^{\om}_0(N/2) - x^{\om}_0(0) > 0.$
\end{lm}

Our approach is to show that in the degenerate case, $\zo$ becomes a collinear solution, which contains at least one isolated collision. Then by local deformations near the isolated collision, we can find another loop from $\lo$, whose action value is strictly smaller than $\zo$'s, which gives us a contradiction. The proof is technical and lengthy, which will be given in the Appendix.   

With the above two lemmas. The fact that $\zo$ satisfies the strictly monotone constraints will be established by the following two lemmas. 

\begin{lm}
\label{lm: even strict} When $N=2n$, for any $\om \in \Om_N$, if $\zo \in \lo$ is an action minimizer of $\ac$ in $\lo$, then $x^{\om}_0(t_1) < x^{\om}_0(t_2)$ for any $0 \le t_1 < t_2 \le N/2$. 
\end{lm}

\begin{proof}
For simplicity let $z =\zo$. By Lemma \ref{nondecreasing}, it is enough to show that, 
$$ x_0(j/2 +t_1) \ne x_0(j/2 +t_2), \;\; \text{ for any } 0 \le t_1 < t_2 \le 1/2, \;\; \forall j \in \N. $$
By a contradiction argument, let's assume there exist $j_0 \in \N$, and $ 0 \le t_1 < t_2 \le 1/2$, such that 
\begin{equation}
\label{eq: x0 equal} x_0(j_0/2 +t) = x_0(j_0/2 +t_1), \;\; \forall t \in [t_1, t_2].
\end{equation}

By Lemma \ref{degenerate}, there exist $t_0, \dl_1>0$ small enough, such that 
\begin{equation}
\label{eq: dl1} x_n(t) - x_0(t) \ge \dl_1, \;\; \forall t \in [0, t_0].
\end{equation}
Meanwhile we can always find a $\dl_2>0$, such that 
\begin{equation}
\label{eq: dl2} |z_n(t) - z_0(t)| \le \dl_2^{-1}, \;\; \forall t \in [0, 1/2].
\end{equation}
Depending on the value of $j_0$, two different cases need to be considered. 

\emph{Case 1}: $j_0$ is even ($j_0=2k$). Then \ref{eq: x0 equal} implies 
$$ x_k(t) = x_k(t_1), \;\; \forall t \in [t_1, t_2].$$
Choose a $\ep>0$ small enough, we define a new loop $\ze \in \lo$ as following:
\begin{equation}
\label{eq: ze 1} \ze_k(t) = \begin{cases}
z_k(t) -\ep, \;\; & \forall t \in [0, t_1], \\
z_k(t)- \frac{t_2-t}{t_2-t_1} \ep, \;\; & \forall t \in [t_1, t_2], \\
z_k(t), \;\; &\forall t \in [t_2, 1/2];
\end{cases}
\end{equation}
\begin{equation}
\label{eq: ze 2} \ze_j(t) = \begin{cases}
z_j(t), \;\; &\text{if } j \in \{k+1, \dots, N-1-k \}, \\
z_j(t)-\ep, \;\; &\text{if } j \in \N \setminus \{k, \dots, N-1-k\},
\end{cases}
\;\; \forall t \in [0, 1/2]. 
\end{equation}
By the above definition of $\ze$ and \eqref{eq: x0 equal}, 
\begin{equation}
\label{eq: AK} \ac_K(\zd^{\ep}) - \ac_k(\zd) = \ey \int_{t_1}^{t_2} |\zd^{\ep}(t)|^2 - |\zd|^2 \, dt = \ey \int_{t_0}^{t_1} \frac{\ep^2}{(t_2-t_1)^2} \,dt = \frac{\ep^2}{2(t_2-t_1)}. 
\end{equation}
This shows the change in kinetic energy. For potential energy, notice that 
\begin{equation}
\label{eq: ze ge z} |\ze_j(t) - \ze_l(t)| \ge |z_j(t)- z_l(t)|, \;\; \forall t \in [0, 1/2], \; \forall \{j \ne l\} \subset \N. 
\end{equation}
When $k >0$, by the definition of $\ze$ and \eqref{eq: dl1}, for any $t \in [0, t_0]$, 
\begin{align*}
|\ze_n(t)-\ze_0(t)|^2 & = |x_n(t) -x_0(t) + \ep + i(y_n(t)-y_0(t))|^2 \\
   					  & = |z_n(t) -z_0(t)|^2 + 2(x_n(t)-x_0(t))\ep + \ep^2 \\
   					  & \ge |z_n(t) -z_0(t)|^2 + 2\dl_1 \ep + \ep^2.
\end{align*}		  
Together with \eqref{eq: dl2}, for any $t \in [0, t_0]$ 
\begin{align*}
|\ze_n(t) & -\ze_0(t)|^{-1}- |z_n(t)-z_0(t)|^{-1} \\
		  & \le \frac{1}{|z_n(t)-z_0(t)|} \Big( ( 1+ \frac{2\dl_1 \ep}{|z_n(t)-z_0(t)|^2}+\frac{\ep^2}{|z_n(t)-z_0(t)|^2})^{-\ey}-1 \Big) \\
		  & \le -\frac{\dl_1 \ep}{|z_n(t)-z_0(t)|^3} + o(\ep) \le -\dl_1 \dl_2^3 \ep + o(\ep).
\end{align*}
Combine this with \eqref{eq: ze ge z}, 
\begin{align}
\label{eq: AU 1} \begin{split}
\ac_U(\ze) - \ac_U(z) & \le \int_0^{t_0} |\ze_n(t) - \ze_0(t)|^{-1} - |z_n(t) - z_0(t)|^{-1} \,dt \\
					  & \le \int_0^{t_0} -\dl_1 \dl_2^3 \ep + o(\ep) \, dt = -C_1(\dl_1, \dl_2, t_0) \ep + o(\ep). 
\end{split}
\end{align}

When $k=0$, if $t_0 < t_1$, the above estimates for potential energy will still hold. However if $t_0> t_1$, then things are slightly different. Now let's assume $k=0$ and $t_0> t_1$, by \eqref{eq: dl1}, for any $t \in [t_1, t_3]$, where $t_3 = \min\{t_0, t_2\}$, 
\begin{align*}
 |\ze_n(t)-\ze_0(t)|^2 & = |x_n(t) -x_0(t) + \frac{t_2-t}{t_2-t_1}\ep + i(y_n(t)-y_0(t))|^2\\
                       & = |z_n(t)-z_0(t)|^2 + 2(x_n(t)-x_0(t))\frac{t_2-t}{t_2-t_1}\ep + \big(\frac{t_2-t}{t_2-t_1}\big)^2 \ep^2 \\
                       & \ge |z_n(t)-z_0(t)|^2 +2\dl_1\ep\frac{t_2-t}{t_2-t_1} + o(\ep).
 \end{align*}	 
Combine the above with \eqref{eq: dl2}, we get, for any $t \in [t_1, t_3]$,
\begin{align*}
|\ze_n(t) & -\ze_0(t)|^{-1}- |z_n(t)-z_0(t)|^{-1} \\
          & \le \frac{1}{|z_n(t)-z_0(t)|} \Big( \big(1+ \frac{2\dl_1 \ep \frac{t_2-t}{t_2-t_1}}{|z_n(t)-z_0(t)|^2}\big)^{-\ey} -1 \Big) \\
          & \le -\frac{\dl_1 \ep \frac{t_2-t}{t_2-t_1}}{|z_n(t)-z_0(t)|^3} + o(\ep) \le -\dl_1 \dl_3^3 \ep \frac{t_2-t}{t_2-t_1}+ o(\ep).
\end{align*}
As a result 
\begin{align}
\label{eq: AU 2} \begin{split}
\ac_U(\ze) &- \ac_U(z)  \le \int_{t_1}^{t_3} |\ze_n(t) -\ze_0(t)|^{-1}- |z_n(t)-z_0(t)|^{-1} \,dt \\
                      & \le \int_{t_1}^{t_3} -\dl_1 \dl_3^3 \ep \frac{t_2-t}{t_2-t_1}+ o(\ep) \, dt \le -C_2(\dl_1, \dl_2, t_0, t_1, t_2)\ep + o(\ep).
\end{split}
\end{align}
Together \eqref{eq: AK}, \eqref{eq: AU 1} and \eqref{eq: AU 2} show that, for $\ep$ small enough 
\begin{equation*}
\ac(\ze) - \ac(z) \le -C_3(\dl_1, \dl_2, t_0, t_1, t_2)\ep + \frac{\ep^2}{2(t_2 -t_1)} + o(\ep) <0,
\end{equation*}
which is a contradiction. This finishes our proof of \emph{Case 1}. 

\emph{Case 2}: $j_0$ is odd ($j_0= 2k-1$). Then \ref{eq: x0 equal} implies 
$$ x_{N-k}(t) = x_{N-k}(1/2-t_1), \;\; \forall t \in [1/2-t_2, 1/2-t_1].$$
For $\ep>0$ small enough, we define a new loop $\ze \in \lo$ as following:

\begin{equation}
\label{eq: ze 1 odd} \ze_{N-k}(t) = \begin{cases}
z_{N-k}(t) +\ep, \;\; & \forall t \in [0, 1/2-t_2], \\
z_{N-k}(t)+ \frac{1/2-t_1-t}{t_2-t_1} \ep, \;\; & \forall t \in [1/2-t_2, 1/2-t_1], \\
z_{N-k}(t), \;\; &\forall t \in [1/2-t_1, 1/2];
\end{cases}
\end{equation}
\begin{equation}
\label{eq: ze 2 odd} \ze_j(t) = \begin{cases}
z_j(t)+\ep, \;\; &\text{if } j \in \{k, \dots, N-1-k \}, \\
z_j(t), \;\; &\text{if } j \in \N \setminus \{k, \dots, N-k\},
\end{cases}
\;\; \forall t \in [0, 1/2]. 
\end{equation}
By estimates similarly as in \emph{Case 1}, we can show $\ac(\ze)-\ac(z)<0$ for $\ep$ small enough, which is a contradiction. 

\end{proof}

\begin{lm}
\label{lm: odd strict} When $N=2n+1$, for any $\om \in \Om_N$, if $\zo \in \lo$ is an action minimizer of $\ac$ in $\lo$, then $x^{\om}_0(t_1) < x^{\om}_0(t_2)$ for any $0 \le t_1 < t_2 \le N/2$.
\end{lm}

\begin{proof} 
Let $z=\zo$. Like the proof of Lemma \ref{lm: even strict}, by a contradiction argument, let's assume \eqref{eq: x0 equal} holds for some $j_0 \in \N$ and $0 \le t_1 < t_2 \le 1/2$.

Notice that by Lemma \ref{nondecreasing} and \ref{degenerate}, 
\begin{equation*}
\max \{ x_0(n) - x_0(0), \; x_0(n+1/2) -x_0(1.2) \} >0.
\end{equation*}
Hence there exist $t_0, \dl_1>0$, such that one of the following must hold
\begin{align}
\label{eq: odd strict 1} x_n(t) -x_0(t) \ge \dl_1, \;\; & \forall t \in [0, t_0];\\
\label{eq: odd strict 2} x_n(t) -x_0(t) \ge \dl_1, \;\; & \forall t \in [1/2-t_0, 1/2].
\end{align}

First let's assume \eqref{eq: odd strict 1} holds. Again we will consider two different cases depending on the value of $j_0$ (to separate them from the two cases consider in the proof of Lemma \ref{lm: even strict}, we will count them as \emph{Case 3} and \emph{Case 4}).  

\emph{Case 3}: $j_0$ is odd ($j_0= 2k-1$). Let $\ze \in \lo$ be defined as in \emph{Case 2}, then a contradiction can reached like there.

\emph{Case 4}: $j_0$ is even ($j_0=2k$). In this case $0 \le k \le n$. Let $\ze \in \lo$ be define as in \emph{Case 1}. 
When $k<n$, a contradiction can be reached by the same argument given there. When $k=n$, more needs to be said. First the set $\{k+1, \dots, N-1-k\}$ in \eqref{eq: ze 2} is empty. By what we have proven so far in \emph{Case 4} and in \emph{Case 3}, there is a $\dl_3>0$, such that 
\begin{equation} 
\label{eq: odd dl 3} x_n(t) - x_0(t) \ge \dl_3, \;\; \forall t \in [0, 1/2]. 
\end{equation}
With \eqref{eq: dl2} still holds, by similarly computations given in \emph{Case 1},
\begin{align*}
\ac_U(\ze) -\ac_U(z) & \le \int_{t_1}^{t_2} |\ze_n(t)-\ze_0(t)|^{-1} - |z_n(t)-z_0(t)|^{-1}\,dt \\
                 & \le -C_4(t_1, t_2, \dl_1, \dl_3)\ep + o(\ep).
\end{align*}
Since the change in kinetic energy is still given by \eqref{eq: AK}, we get $\ac(\ze) -\ac(z)<0$ for $\ep$ small enough, which is absurd. This finishes our proof of \emph{Case 4}. 

Now let's assume \eqref{eq: odd strict 2} hold. The proof is almost the same as above. We will not repeat the details here. 

\end{proof}

Let $\zo$ be an action minimizer of $\ac$ in $\lo$ (for simplicity, we set $z =\zo$ in the rest of this section). By Lemma \ref{lm: even strict} and \ref{lm: odd strict}, it satisfies the strictly monotone constraints. As a result, if $z$ is collision-free, it must be a solution of \eqref{N body}. Notice that Lemma \ref{lm: even strict} and \ref{lm: odd strict} already imply that $z(t), t \in (0, 1/2),$ must be collision-free, and the only possible collisions are binary collisions at $t=0$ or $t=1/2$ between certain pairs of masses determined by the symmetric constraints. To be precise, when $t =0$, a binary collision can only happen between the pairs of masses with the following indices
$$ \{1, N-1 \}, \{2, N-2\}, \dots, \{ \nf, N- \nf \}; $$
when $t = 1/2$, between the following pairs
$$ \{0, N-1 \}, \{1, N-2\}, \dots, \{ \nf, N-\nf-1 \}. $$
When $N=2n+1$, $\nf = N-\nf -1 = n$, so in this case, there is no collision between $m_{\nf}$ and $m_{N-\nf-1}$ at $t=1/2$.

In the following, we will show none of the above binary collisions can exist. First let's assume $z_j(0) = z_{N-j}(0)$, for some $j \in \{1, \dots, \nf\}$. Notice that $y_j(0) = y_{N-j}(0)=0$, as $z_j(0) = \zb_{N-j}(0)$. Without loss of generality, we may further assume $z_j(0)= z_{N-j}(0)=0$. 

For any $t \in [0, 1/2]$ and $k \in \{j, N-j \}$, we set
\begin{align}
\label{zhat} \zh(t) &:= \xh(t)+i \hat{y}(t) := (z_j(t) + z_{N-j}(t))/2,\\
\label{wcom} w_k(t) &:= u_k(t) +i v_k(t) := z_k(t)- \zh(t). 
\end{align}
Here $\zh(t)$ is the center of mass of $m_j$ and $m_{N-j}$, and $w_k(t)$ is the relative position of $m_k$ with respect to $\zh$. Put $w_k(t)$ in polar coordinates
\begin{equation}
\label{wpolar}  w_k(t) = \rho_k(t) e^{i \tht_k(t)},
\end{equation}
we have the following two results.
\begin{prop}
\label{prop: sudman} For any $k \in \{j, N-j\}$, when $t>0$ is small enough
$$ \rho_k(t) = C_1 t^{\frac{2}{3}} + o(t^{\frac{2}{3}}), \;\; \dot{\rho}_k(t)= C_2 t^{-\frac{1}{3}} + o(t^{-\frac{1}{3}}).$$
\end{prop}
This is the well-known Sundman's estimate, for a proof see \cite[(6.25)]{FT04}.

\begin{prop} \label{prop: angle limit}
For any $k \in \{ j, N-j \}$, there exist finite $\tht_k^+$, satisfying the following: 
\begin{equation}
\label{ang lim} \lim_{t \to 0^+} \tht_k(t) = \tht_k^+, \;\; \lim_{ t \to 0^+} \dot{\tht}_{k}(t) = 0, \;\; \tht_{N-j}^+= \tht_j^+ +\pi.
\end{equation}
\end{prop}
The above proposition is also well-known. It shows that $m_j$ and $m_{N-j}$ approach to the binary collision from two definite directions that are opposition to each other (a proof can be found in \cite[Proposition 4.2]{Y16b}). 

By the strictly monotone constraints, $x_j(t) - x_{N-j}(t) >0$, $\forall t \in (0, 1/2]$. Hence
$$ x_{N-j}(t) - \xh(t) <0 <  x_j(t) - \xh(t), \;\; \forall t \in (0, 1/2].$$
As a result, we may assume
$$ \tht_j(t) \in (-\pi/2, \pi/2), \;\;  \tht_{N-j}(t) \in (\pi/2, 3\pi/2),\;\; \forall t \in (0, 1/2],$$
and this means
$$ \tht_j^+ \in  [-\pi/2, \pi/2], \;\; \tht_{N-j}^+ \in [\pi/2, 3\pi/2]. $$ 

Depending on the values of $\tht_j^+$, two types of deformations will be applied to $z$ to get a contradiction. To make sure the loop we obtained after the deformation is still contained in $\lo$, we need to know the precise value of $\om_{2j}$. Without loss of generality, let's assume $\om_{2j}=1$ in the following. 

\begin{lm}
 \label{thtp1} If $z_j(0)=z_{N-j}(0)$ and $\tht_j^+ \in (-\pi/2, \pi/2]$, then there is a $\ztl \in \lo$ with $\ac(\ztl) < \ac(z).$
\end{lm}

To prove the above lemma, we need the following local deformation result near an isolated binary collision.

\begin{prop}
\label{bicoll}
 If $z_j(0)=z_{N-j}(0)$ and $\tht_j^+ \in (-\pi/2, \pi/2]$, then for $\ep>0$ and $t_0=t_0(\ep)>0$ small enough, there is an $\ze \in H^1([0, 1/2], \cc^{N})$ (a local deformation of $z$ near $t=0$) satisfying $\ac(\ze) < \ac(z)$ and the following:
 \begin{enumerate}
 \item[(a).] For any $l \in \N \setminus \{j, N-j\}$, $\ze_l(t) = z_l(t)$, $\forall t \in [0, 1/2]$, 
 \item[(b).] For any $k \in \{j, N-j\}$, 
 $$ \begin{cases}
 \ze_k(t) = z_k(t), \;\; & \text{ when } t \in [t_0, 1/2]; \\
 |\ze_k(t) - z_k(t)| \le \ep, \;\; & \text{ when } t \in [0, t_0], 
 \end{cases} $$
 \item[(c).] $\ze(t)$, $t \in (0, 1/2)$, is collision-free, 
 \item[(d).] $\ze_{k}(0) \ne \ze_l(0)$, for any $k \in \{j, N-j\}$ and $l \in \N \setminus \{j, N-j\}$, 
 \item[(e).] $\ze_j(0) \ne \ze_{N-j}(0)$, in particular 
 $$\ze_j(0)= z_j(0) +i\ep, \;\; \ze_{N-j}(0) = z_{N-j}(0) - i \ep.$$

 \end{enumerate}
\end{prop}
\begin{rem}
 A proof of the above proposition can be found in \cite[Proposition 4.3]{Y16b}. The proof essentially relies on the following fact of the Kepler problem: the parabolic collision-ejection solution connecting two different points with the same distance to the origin has action value strictly smaller than the direct and indirect Keplerian arcs joining them (with the same transfer time). This result was attributed to C. Marchal in \cite{C05}. A proof can be found in \cite{FGN11} and \cite{Y16a}. 
\end{rem}

\begin{proof} 

 [\textbf{Lemma \ref{thtp1}}] 
 By Proposition \ref{bicoll}, after a local deformation of $z$ near the isolated collision $z_j(0)=z_{N-j}(0)$, we get a new path $\ze \in H^1([0, 1/2], \cc^{N})$ with $\ac(\ze) < \ac(z).$ After apply the action of the dihedral group $D_N$ defined before, we get a loop, which will till be denoted by $\ze$. Notice that as a loop $\ze$ is contained in $\Lmd$ and satisfies the $\om$-topological constraints. However it is not so clear whether it is also contained in $\lo$, as the monotone constraints may be violated after the local deformation. Nevertheless we will show that by further modification of $\ze$, we can always get a $ \ztl \in \lo$ satisfying $\ac(\ztl) \le \ac(\ze) < \ac(z)$. 

 Recall that $z_j(0)=z_{N-j}(0)=0$. By Proposition \ref{bicoll} (e), $x^{\ep}_j(0) = x^{\ep}_{N-j}(0)=0$. Let $t_0$ be given as in Proposition \ref{bicoll}, we define $\dl_1, \dl_2$ as following
 $$ \dl_1 = - \min \{ x^{\ep}_j(t) : t \in [0, t_0] \}, \;\; \dl_2 = \max \{ x^{\ep}_{N-j}(t): t \in [0, t_0] \}. $$
 Then $\dl_1 \ge 0$, and so is $\dl_2$. Furthermore let 
 $$ t_1 = \min \{ t \in [0, t_0]: x_j^{\ep}(t) = -\dl_1 \}, \;\; t_2 = \min \{ t \in [0, t_0]: x_{N-j}^{\ep}(t) = \dl_2 \}$$
 and 
 $$ \mathbb{T}_1 =\{ t \in [0, t_1]: x_j^{\ep}(t) <0 \}, \;\; \mathbb{T}_2 = \{ t \in [0, t_2]: x_{N-j}^{\ep}(t) > 0 \}. $$ 
 
Now we define a new path $\ztl(t) = (\ztl_k(t))_{k \in \N}$ as following:   
$$  \ztl_j(t) =
\begin{cases}
 \ze_j(t) & \text{ if } t \in [0, t_1] \setminus \mb{T}_1, \\
 -\bar{\ze}_j(t) & \text{ if } t \in \mb{T}_1, \\
 \ze_j(t) + 2 \dl_1, & \text{ if } t \in [t_1, 1/2], 
\end{cases}
$$
$$ \ztl_{N-j}(t) = 
\begin{cases}
 \ze_{N-j}(t), & \text{ if } t \in [0,t_2] \setminus \mb{T}_2, \\
 - \bar{\ze}_{N-j}(t), & \text{ if } t \in \mb{T}_2, \\
 \ze_{N-j}(t) - 2 \dl_2, & \text{ if } t \in [t_2, 1/2], 
\end{cases} 
$$
$$ \ztl_k(t) = 
\begin{cases}
\ze_k(t) + 2 \dl_1, & \text{ if } k \in \{j+1, \dots, N-j -1 \}, \\
\ze_k(t) - 2 \dl_2, & \text{ if }  k \in \N \setminus \{j, \dots, N-j \},
\end{cases}
\;\; \forall t \in [0, 1/2].
$$

We point out that for any $k \in \{j, N-j\}$, if $\dl_k=0$, then $t_k=0$ and $\mb{T}_k=\emptyset$. In particular, when $\dl_1 =\dl_2=0$, then $\ztl = \ze$. See Figure \ref{fig: C4} for illuminating pictures of $\ztl_j$ (or $\ztl_{N-j}$), when $\dl_1 >0$ (or $\dl_2 >0$). 

For $\ztl$ defined as above, the monotone constraints are satisfied and $\ztl \in \lo$. Furthermore $\ac(\ztl) \le \ac(\ze) < \ac(z),$ as $\ac_K(\dot{\ztl})= \ac_K(\zd^{\ep})$, and $\ac_U(\ztl) \le \ac_U(z^{\ep})$, where the estimate on potential energy follows from
$$ |\ztl_k(t) - \ztl_l(t)| \ge |\ze_k(t) - \ze_l(t)|, \;\; \forall t \in [0, 1/2], \;\; \forall \{k < l \} \subset \N. $$

\begin{figure}
	\centering
	\includegraphics[scale=0.7]{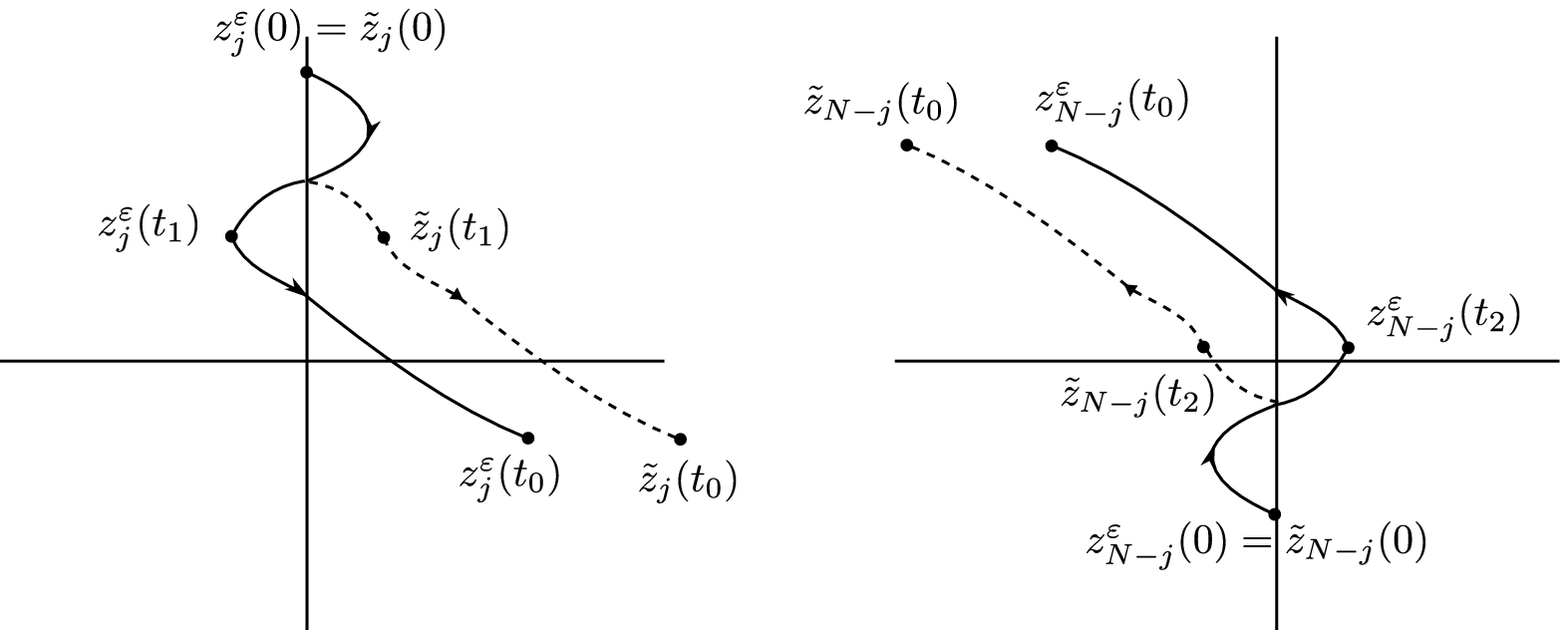}
	\caption{}
	\label{fig: C4}
\end{figure}

\end{proof}

We point out that Lemma \ref{thtp1} does not apply when $ \tht^+_j = -\pi/2$. By Gordon's result on Kepler problem (see \cite{Go77}), an argument in the nature of local deformation can never rule out the binary collision in this case. Some type of global estimate has to be involved, which generally is hard to do. The advantage of our approach is that the necessary global estimate can be obtained based on the monotone constraints. This will be given in the next lemma. 

\begin{lm}
 \label{thtp2} If $z_j(0)=z_{N-j}(0)$ and $\tht_j^+=-\pi/2$, then for $\ep_1>0$ small enough, there is a $\zey  \in \lo$ with $\ac(\zey) < \ac(z).$
\end{lm}

\begin{proof}
	For $\ep_1>0$ small enough, define a new loop $\zey \in \lo$ as following:
	$$ \zey_j(t) = \begin{cases}
	z_j(t) + t(2\ep_1 -t), \;\; & \forall t \in [0, \ep_1], \\
	z_j(t) + \ep_1^2, \;\; & \forall t \in [\ep_1, 1/2], 
	\end{cases} $$
	$$ \zey_{N-j}= \begin{cases}
	z_{N-j}(t)-t(2\ep_1 -t), \;\; & \forall t \in [0, \ep_1], \\
	z_{N-j}(t)-\ep_1^2, \;\; & \forall t \in [\ep_1, 1/2], 
	\end{cases} $$
	$$ \zey_k(t) = \begin{cases}
	z_k(t) + \ep_1^2, \;\; & \text{ if } k \in \{j+1, \dots, N-j-1 \}, \\
	z_k(t) - \ep_1^2, \;\; & \text{ if } k \in \N \setminus \{j, \dots, N-j\},
	\end{cases} \;\; \forall t \in [0, 1/2]. $$

	We claim $\ac(\zey) < \ac(z)$ for $\ep_1$ small enough. The above definition immediately implies
	$$ \ac_K(\zd^{\ep_1}) > \ac_K(\zd), \;\; \ac_U(\zd^{\ep_1}) < \ac_U(z). $$
	However to get the desired result, the above estimates have to be improved. 

	For $\ac_K$, we need the estimates of $\dot{x}_j(t)$ and $\dot{x}_{N-j}(t)$, when $t$ goes to zero. Let $\zh$ and $w_k$, $k \in \{j, N-j\}$, be defined as in \eqref{zhat} and \eqref{wcom}. Then $\dot{x}_k(t) = \dot{u}_k(t) + \dot{\hat{x}}(t)$. Meanwhile by \eqref{wpolar}, 
  		$$ \dot{u}_k = \dot{\rho}_k \cos \tht_k - \rho_k \dot{\tht}_k \sin \tht_k, \;\; \forall k \in \{j, N-j\}.$$
    As $ \tht_j^+= \lim_{t \to 0^+}\tht_j(t)=-\pi/2$, by Proposition \ref{prop: sudman} and \ref{prop: angle limit}, we get 
     \begin{equation}
     \label{uj} |\dot{u}_j(t) | \le C_1 t^{\frac{2}{3}} \text{ for } t>0 \text{ small enough. }
     \end{equation}
    Since $\tht_{N-j}^+= \tht_j^+ \pi = \pi/2$, similarly we have
    \begin{equation}
    \label{unj} |\dot{u}_{N-j}(t)| \le C_2 t^{\frac{2}{3}} \text{ for } t>0 \text{ small enough. }
    \end{equation}
        
    Notice that although $m_j$ and $m_{N-j}$ collide at $t=0$, there center of mass $\zh(t)$ is $C^2$ at $t=0$, and we claim
    \begin{equation*}
      \dot{\xh}(0) = 0.
    \end{equation*} 
    Otherwise let's assume $\dot{\xh}(0) <0.$ Then by \eqref{uj},
     $$ \dot{x}_j(t) = \dot{u}_j(t) + \dot{\hat{x}}(t)  < 0, \text{ for } t >0 \text{ small enough}.$$ 
    However this violates the monotone constraints. Similarly if $\dot{\xh}(0) > 0$, then 
    $$ \dot{x}_{N-j}(t)= \dot{u}_{N-j}(t) + \dot{\hat{x}}(t) >0, \text{ for } t> 0 \text{ small enough},$$
    which is again a contradiction to the monotone constraints. This proves our claim. As a result,
    \begin{equation}
    \label{xsd} |\dot{\xh}(t)| \le C_3 t \text{ for } t>0 \text{ small enough. }
    \end{equation} 
    Combine this with \eqref{uj} and \eqref{unj}, we get
  \begin{equation}
   \label{hori vel} 
   |\dot{x}_k(t) | \le C_4 t^{\se} \text{ for } t>0 \text{ small enough,} \;\; \forall k \in \{j, N-j\}. 
  \end{equation}
   
   By the definition of $\zey$ and \eqref{hori vel}, 
   \begin{equation}
   \label{eq: K difference} \begin{split}
   \ac_K(\zey) - \ac_K(z) &= \ey \int_0^{\ep_1} \sum_{k \in \{j, N-j\}} \big( |\zd_k^{\ep_1}| - |\zd_k|^2 \big) \,dt \\
   						  &= \int_0^{\ep_1} 4(\ep_1 -t)^2 +2 (\ep_1-t)\dot{x}_j(t) - 2 (\ep_1-t)\dot{x}_{N-j}(t) \,dt \\
   						  & \le 4\int_0^{\ep_1} (\ep_1 -t)^2 + C_4 t^{\frac{2}{3}}(\ep_1-t) \,dt \le C_5 \ep_1^{\frac{8}{3}} + o( \ep_1^{\frac{8}{3}}). 
   \end{split}	
   \end{equation}

   Now we will estimate the change in potential energy. By Lemma \ref{lm: even strict} and \ref{lm: odd strict}, $x_j(1/2)-x_{N-j}(1/2) >0$. Then for any $\dl>0$ small enough, there exist positive constants $C_6, C_7$ (both independent of $\ep_1$), such that 
   \begin{equation}
   \label{eq: xj xN-j} x_j(t) -x_{N-j}(t) \ge C_6, \;\; |z_j(t)-z_{N-j}(t)|^{-1} \ge C_7, \;\; \forall t \in [1/2-\dl, 1/2].
   \end{equation}
   Meanwhile for any $t \in [1/2-\dl, 1/2]$, 
   $$ |\zey_j(t)-\zey_{N-j}(t)|^{-1} = \big(|z_j(t)-z_{N-j}(t)|^2 +4\ep_1^2(x_j(t)-x_{N-j}(t))+4\ep_1^4 \big)^{-\ey}.$$
   By \eqref{eq: xj xN-j} and a simple computation, we get, for any $t \in [1/2-\dl, 1/2]$, 
   \begin{equation*}
   \frac{1}{|\zey_j(t)-\zey_{N-j}(t)|} - \frac{1}{|z_j(t)-z_{N-j}(t)|}  \le \frac{-2(x_j(t)-x_{N-j}(t))}{|z_j(t)-z_{N-j}(t)|^3}\ep_1^2 +o(\ep^2). 
   \end{equation*}
   Notice the definition of $\zey$ implies 
   \begin{equation}
   |\zey_k(t) - \zey_l(t)| \ge |z_k(t) - z_l(t)|, \;\; \forall t \in [0, 1/2], \;\; \forall \{k \ne l \} \subset \N. 
   \end{equation}
   As a result 
   $$ \ac_U(\zey)- \ac_U(z) \le \int_{\ey-\dl}^{\ey} \frac{1}{|\zey_j-\zey_{N-j}|}-\frac{1}{|z_j-z_{N-j}|} \,dt \le -C_8 \dl \ep_1^2. $$
   Combining this with our estimate on $\ac_K$ obtained earlier, we get 
   $$ \ac(\zey) - \ac(z) \le -C_8\dl \ep^2 + C_5 \ep_1^{\frac{8}{3}} < 0 $$
   for $\ep_1$ small enough, as $C_5, C_8$ are independent of $\ep_1$.

\end{proof}

By Lemma \ref{thtp1} and Lemma \ref{thtp2}, we have proven there is no collision between $m_j$ and $m_{N-j}$ at $t=0$, for any $j \in \{1, \dots, \mf{n} \}$. Similarly it can be proven that there is no collision between $m_j$ and $m_{N-1-j}$ at $t=1/2$, for any $j \in \{0, \dots, \mf{n} \}$. We will not repeat it here. As a result, we have proven the following. 
\begin{prop}
For any $\om \in \Om_N$, the action functional $\ac$ has at least one minimizer $\om \in \lo$. Furthermore every action minimizer $\zo$ satisfies the strictly monotone constraints, and is a collision-free solution of \eqref{N body}.\end{prop}

However we still haven't proven \eqref{eq: dot x > 0} in Theorem \ref{thm 1}. This will be demonstrated by the next lemma. Notice this is not necessarily true, even when the strictly monotone constraints are satisfied. 
 \begin{lm}
 	\label{strictly increasing} For any $\om \in \Om_N$, if $\zo \in \lo$ is a minimizer of $\ac$ in $\lo$, then it satisfies \eqref{eq: dot x > 0}.
 \end{lm}
 
 \begin{proof}
       For simplicity, let $z =\zo$. We give the details for $N=2n$ (for $N=2n+1$, it can be proven similarly). By \eqref{eq: z0 zj zN-j}, it is equivalent to prove the following: 
       \begin{align}
       \label{eq: x dot >0 1} & \dot{x}_0(0)= \dot{x}_{\mf{n}+1}(0)=0,   \\
       \label{eq: x dot >0 2} & \begin{cases}
       \xd_0(t) >0, \;\; & \forall t \in (0, 1/2], \\
       \xd_j(t) >0, \;\; & \forall t \in [0, 1/2], \; \forall j \in \{1, \dots, \mf{n} \}, 
       \end{cases}   \\
       \label{eq: x dot >0 3} & \begin{cases}
       \xd_{\mf{n}+1}(t)<0, \;\; & \forall t \in (0, 1/2], \\
       \xd_{j}(t) < 0, \;\; & \forall t \in [0, 1/2], \; \forall j \in \{\mf{n}+2, \dots, N-1 \}. 
       \end{cases}           
       \end{align}
       Since $z$ is a collision-free minimizer of $\ac$ in $\lo$, $z_0(t)$ and $z_{\mf{n}+1}(t)$ must be perpendicular to the real axis at $t =0$, and \eqref{eq: x dot >0 1} follows immediately.

       Now let's prove \eqref{eq: x dot >0 2} (the proof of \eqref{eq: x dot >0 3} is similarly). By a contradiction argument, assume it does not hold. Then by the strictly monotone constraints, $\xd_j(t_0)=0$, for some $t_0 \in [0, 1/2]$ and $j \in \{0, \dots, \mf{n}\}$ (if $t_0=0$, $j \ne 0$). Depending on the value of $t_0$, three different cases will be considered. 

       \emph{Case 1}: $t_0=0$. For $\ep_1>0$ small enough, let $\zey \in \lo$ be the same path defined in the proof of Lemma \ref{thtp2}. Then
       \begin{equation}
        \label{eq: ac U differe x dot}  \ac_U(\zey) - \ac_U(z) \le -C_1\ep_1^2. 
        \end{equation}
        Meanwhile as $\xd_j(0)=\xd_{N-j}(0)=0$ (by the symmetric constraints, $\xd_{N-j}(0)= -\xd_{j}(0)$), we have 
        $$ |\xd_k(t)| \le C_2 t \text{ for } t >0 \text{ small enough}, \;\; \forall k \in \{j, N-j\}. $$ 
        This in fact is a better estimate than \eqref{hori vel}. Then by a computation similar to \eqref{eq: K difference}, we get 
        \begin{equation}
        \ac_K(\zey) - \ac_K(z) \le C_3 \ep_1^3.
        \end{equation}
        As a result, $\ac(\zey)-\ac(z)<0$ for $\ep_1>0$ small enough, which is a contradiction. This finishes our proof of \emph{Case 1}. 

        For the remaining two cases, similar estimates as above will give us contradictions as well. We just give the definition of $\zey \in \lo$ and omit the details. 

        \emph{Case 2}: $t_0=1/2$. Let $\zey \in \lo$ be defined as following: 
        $$ \zey_j(t) = \begin{cases}
        \zey_j(t) - \ep_1^2, \;\; & \forall t \in [0, 1/2-\ep_1], \\
        \zey_j(t) - (1/2-t)(2\ep_1-(1/2-t)), \;\; & \forall t \in [1/2-\ep_1, 1/2], 
        \end{cases} $$
        $$ \zey_{N-j-1}(t) = \begin{cases}
        \zey_{N-j-1}(t) +\ep_1^2, \;\; & \forall t \in [0, 1/2-\ep_1], \\
        \zey_{N-j-1}(t) +(1/2-t)(2\ep_1-(1/2-t)), \;\; & \forall t \in [1/2-\ep_1, 1/2], 
        \end{cases} $$
        $$ \zey_k(t) = \begin{cases}
        \zey_k(t)+\ep_1^2, \; &\text{ if } k \in \{j+1, \dots, N-j-2\}, \\
        \zey_k(t)-\ep_1^2, \; &\text{ if } k \in \N \setminus \{j, \dots, N-j-1\}, 
        \end{cases} \;\; \forall t \in [0, 1/2]. $$

        \emph{Case 3}: $t_0 \in (0, 1/2)$. Let $\zey \in \lo$ be defined as following:
        $$ \zey_j(t) = \begin{cases}
        \zey_j(t) - \ep_1^2, \;\; & \forall t \in [0, t_0-\ep_1], \\
        \zey_j(t) + (t-t_0)(2 \ep_1 - |t- t_0|), \;\; & \forall t \in [t_0-\ep_1, t_0+\ep_1], \\
        \zey_j(t) + \ep_1^2, \;\; & \forall  t \in [t_0+\ep_1, 1/2],  
        \end{cases} $$
        $$ \zey_k(t) = \begin{cases}
        \zey_k(t) + \ep_1^2, \; & \text{ if } k \in \{j+1, \dots, N-j-1 \}, \\
        \zey_k(t) - \ep_1^2, \; & \text{ if } k \in \N \setminus \{j, \dots, N-j-1 \},
        \end{cases} \;\; \forall t \in [0, 1/2]. $$      
       
 \end{proof}
 
 We finish this section with a remark about the existence of these simple choreographies in general homogeneous potentials: 
  $$ U_{\al}(z) = \sum_{\{j < k\} \subset \N} \frac{1}{|z_j-z_k|^{\al}}, \;\; \al >0. $$
 Such a potential is called a \emph{strong force}, if $\al \ge 2$; a \emph{weak force}, if $0 < \al <2.$ The Newtonian potential corresponds to $\al=1$. 

 All the results we proved in this section hold for $U_{\al}$ with $\al \ge 1$, but not when $0 < \al <1$. The reason is that when $\al <1$, Proposition \ref{bicoll} only holds for $\tht_j^+ \in (-\pi/2+\beta(\al), \pi/2]$, for some $\beta(\al)>0$.

\section{Simple choreographies with additional symmetries} \label{sec: addsym}

As we mentioned all the simple choreographies that were proven in Section \ref{sec proof} belong to the family of linear chain, and so are the Figure-Eight solution and the Super-Eight solution. We would like to compare our results with those two. First let's we recall what we know about those two solutions. There are several different figure eight type simple choreographies depending on the symmetric constraints, see \cite{C02} and \cite{FT04}. Here by the Figure-Eight solution, we only mean the one proved by Chenciner and Montgomery in \cite{CM00}. 

\begin{ex} \label{ex1}
      The Figure-Eight solution, $z^e$, is a collision-free minimizer of the action functional $\ac$ in $\Lmd_3^{D_6}$ with the action of the dihedral group $D_6$ defined as
      $$ \uptau(g)t = t-1/2, \quad \rho(g) q = -\bar{q}, \quad \sg(g) = (0,1,2)^2,$$
      $$ \uptau(h)t= -t+1, \quad \rho(h)q = \bar{q}, \quad \sg(h) = (0,2). $$ 
      The corresponding loop $z^e_0 \in H^1(\rr /3\zz, \cc)$ is in the shape of a figure eight symmetric with respect to the origin, the real and imaginary axes. Furthermore it was proved in \cite{KZ03} and \cite{FM05} that each lobe of the eight is convex.
\end{ex}

\begin{ex} \label{ex2}
    The Super-Eight solution, $z^s$, belongs to $\Lmd_4^{D_4 \times \zz_2}$, where $D_4$ is the dihedral group with the action defined as in the section \ref{sec intro},  and $\zz_2 = \langle f|f^2=1 \rangle$ with the action defined as following
	$$ \uptau(f)t=t, \quad \rho(f)q = -q, \quad \sg(f) = (0,2)(1,3).$$
    However the action minimizer of $\ac$ in $\Lmd_4^{D_4 \times \zz_2}$ is not $z^s$, but the rotating $4$-gon.

    Instead $z^s$ is an action minimizer among all $z \in \Lmd_4^{D_4 \times \zz_2}$ satisfying the following topological constraints: 
    	$$ y_0(j/2) = \om_j|y_2(j/2)|, \quad \forall j \in \{1,2,3\},$$
	where $\om = (\om_1, \om_2, \om_3)=(1, -1, 1).$
	
	This was proved analytically by Shibayama in \cite{Sh14} and  he also confirmed numerically that it should look like a super eight.	
\end{ex}

Meanwhile for any $\om \in \Om_N$, let $z^{\om} \in \Lmd_N^{D_N}$ be a simple choreography obtained by Theorem \ref{thm 1}.

\begin{ex} \label{ex3}
 When $N=3$ and $\om = (\om_1, \om_2) = (-1,1)$, an illuminating picture of $z^{\om}$ can be found in Figure \ref{C7}(a), where the solid curves represent the motion of the masses between $t=0$ and $t=1/2$. Such a $z^{\om}$ also looks like a figure eight. The symmetric constraints only implies the loop $z_0$ is symmetric with respect to the real axis. However unlike the Figure-Eight solution, it is not clear whether it is also symmetric with respect to the imaginary axis and the origin.  
\end{ex}

\begin{ex} \label{ex4}
 When $N=4$ and $\om = (\om_1, \om_2, \om_3)=(1, -1, 1)$, an illuminating picture of $z^{\om}$ can be found in Figure \ref{C7}(b), where the solid curves represent the motion of the masses between $t=0$ and $t=1/2$. Such a $z^{\om}$ looks like a super eight symmetric with respect to the real axis, but unlike the Super-Eight solution, we don't if it is also symmetric with respect to the imaginary axis and the origin.  
\end{ex}

\begin{figure}
	\centering
	\includegraphics[scale=0.65]{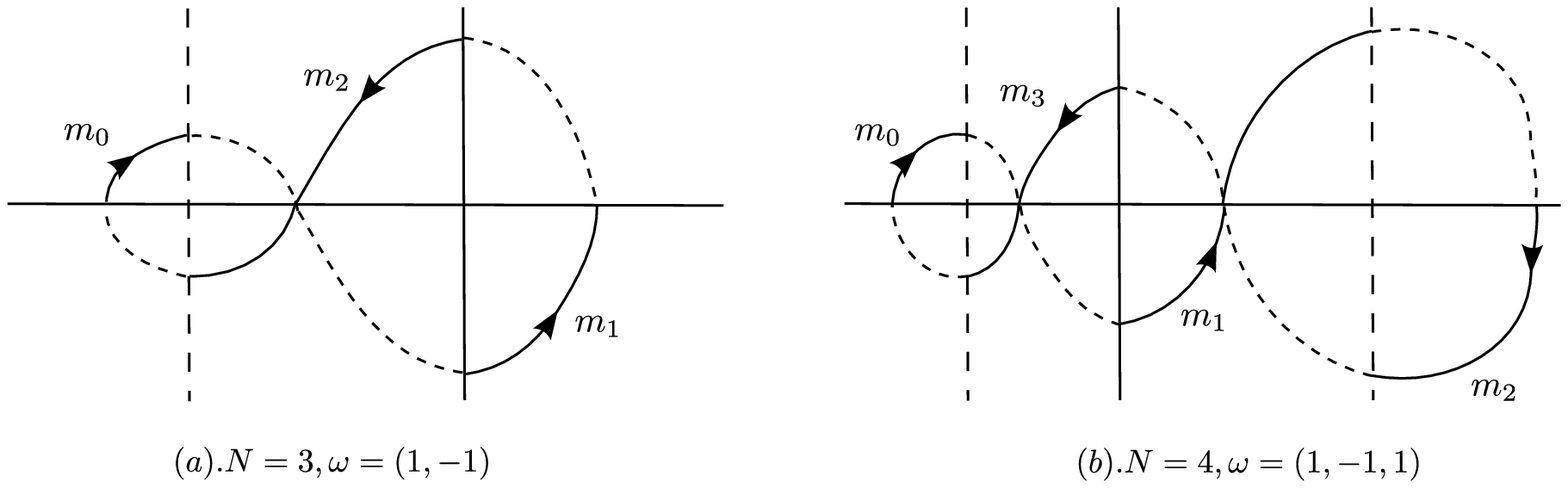}
	\caption{}
	\label{C7}
\end{figure}

The similarities and differences between the above examples inspire us to introduce the symmetric groups $ H_N := D_N \times \zz_2$, where $D_N$ and $\zz_2$ are the same as above, but the action will be defined as following:

if $N=2n$, 
\begin{align}
 \label{even} \begin{split}
              \uptau((g,1))t & = t-1/2, \quad \rho((g,1)) q = \bar{q}, \quad \sg((g,1))=(0, \dots, 2n-1), \\
              \uptau((h,1))t &= 1-t, \quad \rho((h,1)) q = \bar{q}, \quad \sg((h,1)) = (0, 2n-1)\cdots(n-1, n),\\
              \uptau((1,f))t& =t, \quad \rho((1,f))q = -q, \quad \sg((1,f))= (0,n)\cdots(n-1, 2n-1);
             \end{split}
\end{align}

if $N=2n+1$, notice
$$H_{2n+1} \cong D_{4n+2} =\langle (g, f), (h,1) | (g,f)^{4n+2} = (h,1)^2 =1, \big((g,f)(h,1)\big)^2 =1 \rangle, $$
\begin{align}
 \label{odd} \begin{split}
              \uptau((g,f))t & = t-1/2, \quad \rho((g,f)) q = -\bar{q}, \quad \sg((g,f))=(0, \dots, 2n)^{n+1}, \\
              \uptau((h,1))t &= 1-t, \quad \rho((h,1)) q = \bar{q}, \quad \sg((h,1)) = (0, 2n)\cdots(n-1, n+1).
             \end{split}
\end{align}

$D_N$, as a subgroup of $H_N$, has an action induced from $H_N$ as following: 
$$ \uptau((g,1))t = t-(n+1), \;\; \rho((g,1))q = q,  \;\; \sg((g,1))= \sg((g,f)), $$
$$ \uptau((h,1)) t = 1-t, \;\; \rho((h,1)) q = \bar{q}, \;\; \sg((h,1)) = (0, N-1)\cdots(\nf, N-1-\nf).$$
To see the first one notice that $(g,1)= (g,f)^{2n+2}$. This is the same as the action of $D_N$ defined in Section \ref{sec intro}. As a result, $\Lmd_N^{H_N} \subset \Lmd^{D_N}$.

Let $\Lmd_{N}^{H_N, +}$ be the subset of all loops in $\Lmd_N^{H_N}$ satisfies the monotone constraints given in Definition \ref{mc}. For each $\om \in \Om_N$ satisfies 
\begin{equation}
 \label{eq: om} |\om_j - \om_{N-j}| = \begin{cases}
 2, \;\; &\text{ if } N =2n+1, \\
 0, \;\; &\text{ if } N =2n, 
 \end{cases} \;\; \forall j \in \{1, \dots, \mf{n}\}, 
 \end{equation}
 let $\Lmd_{N, \om}^{H_N, +}$ be the subset of all loops in $\Lmd_{N}^{H_N, +}$ satisfies the $\om$-topological constraints. Then by the similar arguments given in Section \ref{sec proof}, we can prove the following result. 
\begin{thm}
 \label{thm 2} For each $\om \in \Om_N$ satisfying \eqref{eq: om}, the action functional $\ac$ has at least one minimizer $\zo \in \Lmd_{N, \om}^{H_N, +}$. Furthermore $\zo$ is collision-free and a simple choreography $z^{\om}(t) \in \hat{\Lmd}_N^{H_N}$ of equation \eqref{N body} satisfying all the properties in Theorem \ref{thm 1} and the following:
 \begin{enumerate}
  \item[(d).] 
  $
   \zo_0(t) = \begin{cases}
   -\zo_0(N/2 -t), \; &\text{ if } N=2n+1, \\
   -\bar{\zo_0}(N/2-t), \; &\text{ if } N=2n, 
   \end{cases} \;\; \forall t \in \rr/N\zz. 
  $
  \end{enumerate}
\end{thm}

\begin{rem}
\begin{enumerate}
\item[(i).] The above theorem can be proved by the same argument used in the proof of Theorem \ref{thm 1}. However during the deformation one needs to make sure the addition symmetric constraints are satisfied afterward. Similar problems and the detailed arguments can be found in another paper by the author \cite{Y16s}.
\item[(ii).] Condition \eqref{eq: om} is implied by the symmetric constraints of $H_{N}$. For those $\om$'s satisfying this condition. It will be interesting to see whether the simple choreographies obtained by Theorem \ref{thm 1} and \ref{thm 2} are actually the same. Similar questions were asked by Chenciner regarding the figure eight type solutions in \cite{C02}.  
\item[(iii).] When $N=3$, $H_3 \cong D_6$ and the action of $H_3$ defined above is identical to the one given in Example \ref{ex1}. Similarly when $N=4$, the action of $H_4$ is the same as the one given in Example \ref{ex2}. Hence for $N$ and $\om$ given in Example \ref{ex3} and \ref{ex4}, the simple choreographies obtained by Theorem \ref{thm 2} are exactly the Figure-Eight solution and the Super-Eight solution.
\end{enumerate}
\end{rem}

As we mentioned,  in \cite{Sh14} the actual shape of the Super-Eight solution were confirmed numerically. In the remaining of this section, we prove the following result proposition, which \emph{almost} confirms the shape of the Super-Eight solution without any numerical result.

\begin{prop}
	\label{super eight} Let $\zo$ be the simple choreography obtained in Theorem \ref{thm 2} with $N=4$ and $\om=(-1,1,-1)$. Then between $t=0$ and $t=1/2$, $\zo_1(t)$, as well as $\zo_3(t)$, has exactly one transversal intersection with the real axis; $\zo_0(t)$, as well as $\zo_2(t)$, has at least one and at most two transversal intersection with the real axis. 
\end{prop}

\begin{proof}
	For simplicity let $z=\zo$. The four masses always form an parallelogram:
	$$ z_3(t) = -z_1(t), \;\; z_0(t) = z_2(t), \;\; \forall t \in \rr. $$
	As a result the kinetic and negative potential energy only depend on $z_1$ and $z_2$:
	$$ K(\zd) = K(\zd_1, \zd_2) = |\zd_1|^2 + |\zd_2|^2;$$
	$$ U(z) = U(z_1, z_2) = \frac{1}{2|z_1|} + \frac{1}{2|z_2|} + \frac{2}{|z_1 - z_2|} + \frac{2}{|z_1 + z_2|}. $$
	By Theorem \ref{thm 2},  
	\begin{equation}
	\label{iin31} y_1(0) < 0 < y_1(1/2), \quad  y_2(0) = 0 > y_2(1/2),
	\end{equation}
	\begin{equation}
	\label{iin32}  0 =x_1(0) < x_1(t) < x_1(1/2) = x_2(1/2) < x_2(t) < x_2(0), \;\; \forall t \in (0, 1/2).
	\end{equation}

   The key to our proof is the following feature of the parallelogram $4$-body problem we borrowed from \cite{Sh14}. For any $j \in \{1,2\}$, let $z_j^* = z_j $ or $\zb_j$, then 
   \begin{equation}
   \label{iin33} U(z_1, z_2) \ge U(z_1^*, z_2^*).
   \end{equation} 
   In particular, if $z_1,z_2$ are in the same quadrant and $z^*_1, z^*_2$ are not, then
   \begin{equation}
   \label{iin34} U(z_1, z_2) > U(z_1^*, z_2^*).
   \end{equation}
   
   By \eqref{iin31}, $z_1(t)$ has at least one transversal intersection with the real axis. We claim this is the only one. Otherwise let's assume $0<t_0<t_1<t_2<1/2$ are the three earliest moments that an transversal intersection occurs. We define two subsets $\mb{T}_1$, $\mb{T}_2$ of $[0, 1/2]$ as following:   
   \begin{align*}
   \text{if } y_2(t_1) & \le 0, \\
   \mb{T}_1 &:= \{ t \in [t_1, 1/2]: y_1(t) <0 \}, \;\; \mb{T}_2:= \{ t \in [t_1, 1/2]: y_2(t) >0 \};\\
   \text{ if } y_2(t_1) &>0,\; y_2(t_0) \le 0, \\
   \mb{T}_1 &:= \{ t \in [t_0, 1/2]: y_1(t) <0 \}, \;\; \mb{T}_2 := \{ t \in [t_0, 1/2]: y_2(t) >0 \};\\
   \text{ if } y_2(t_1) &> 0,\; y_2(t_0) >0, \\
   \mb{T}_1 &:= \{ t \in [t_0, t_1]: y_1(t) >0 \}, \;\; \mb{T}_2:= \{ t \in [t_0, t_1]: y_2(t) <0 \},\\
   \end{align*}
   and a new path $\ztl(t) = (-\ztl_2(t), \ztl_1(t), \ztl_2(t), -\ztl_1(t))$, as following (see Figure \ref{C8})
   $$ \ztl_j(t) = \begin{cases}
   \zb_j(t) \quad  \text{ if } t \in \mb{T}_j;\\
   z_j(t) \quad \text{ if } t \in [0, 1/2] \setminus \mb{T}_j,
   \end{cases} \;\; \forall j \in \{1, 2\}
   $$
   
   \begin{figure}
   	\centering
   	\includegraphics[scale=0.7]{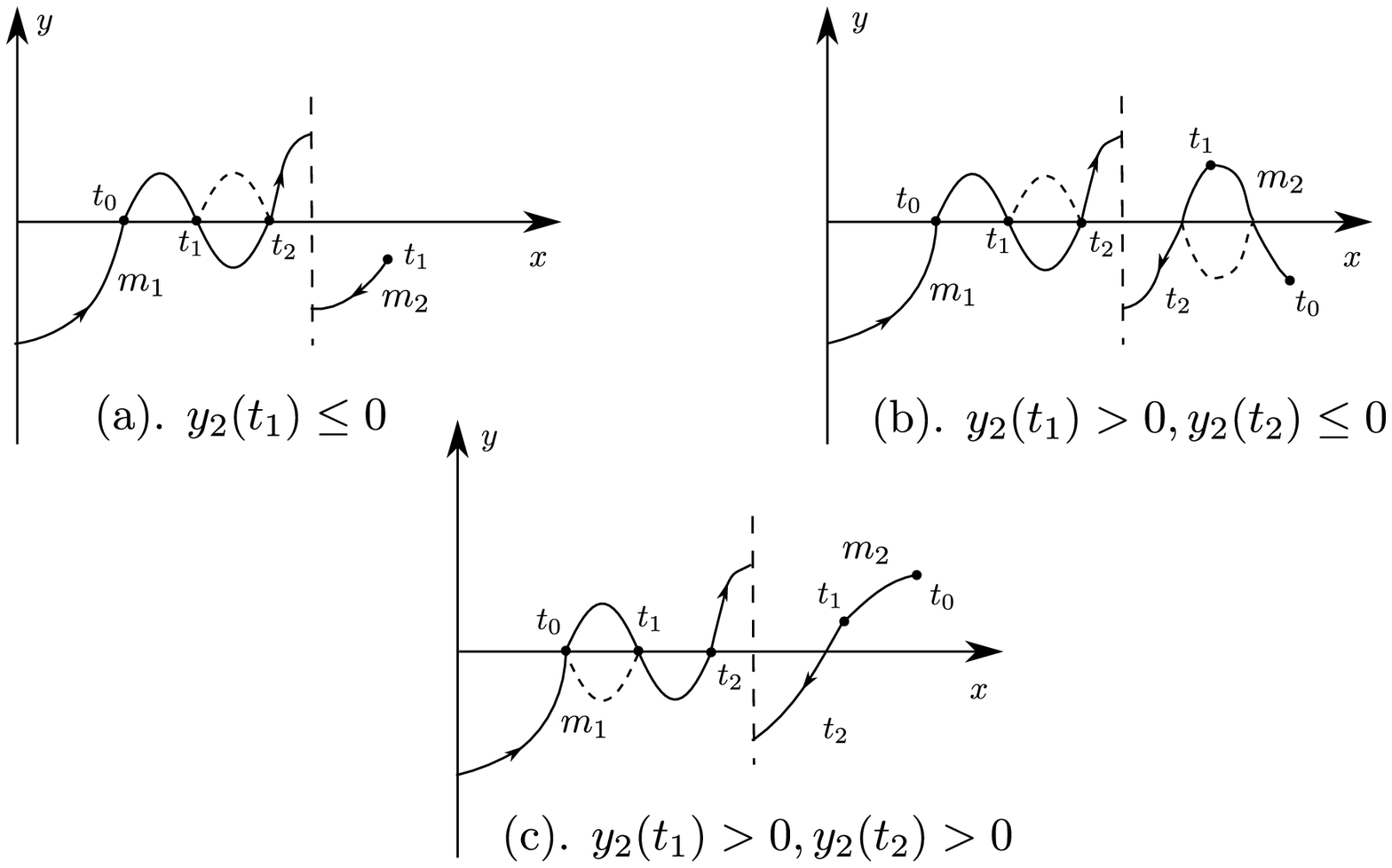}
   	\caption{}
   	\label{C8}
   \end{figure}
   
   By \eqref{iin33}, \eqref{iin34} and the above definition of $\ztl$,
   $$  \int_0^{1/2} K(\zd) \,dt = \int_0^{1/2} K(\dot{\ztl}) \,dt \; \; \int_0^{1/2} U(z_1, z_2) \,dt > \int_0^{1/2} U(\ztl_1, \ztl_2 ) \,dt. $$
   so $\ac(z) > \ac(\ztl)$, which is a contradiction to the minimizing property of $z$. This proves the first half of the proposition. For the second half, as $z_2(t)$ already has an transversal intersection with the real axis at $t=0$, a similar argument as above shows that $z_2(t)$ can have at most one transversal intersection with the real axis besides the one at $t=0$.    
\end{proof}

\section{Appendix}

We give a proof of Lemma \ref{degenerate} in this section. The idea is to combine a local deformation technique near an isolated collision similar to the one introduced by Montgomery in \cite{Mo99} and a blow-up technique introduced by Terracini (see \cite{Ve02}, \cite{FT04} and \cite{Y16b}). In the above references the configurations that serves as the directions of deformation need to be a \emph{centered configuration} (see Definition \ref{centered}).

In our case, we need to make deformations along directions that are not necessarily given by a centered configuration. Without this condition, the action value of the deformed path may not be strictly smaller than the original one. However in our setting, all the masses are traveling on a single vertical line, and we just need to deform along the directions that are orthogonal to this vertical line and the desired result will still hold. 

Let $z \in H^1([0, S], \cc^N)$ be a collision solution satisfying the following conditions: 
\begin{enumerate}
\item[(i).] $z(t)$ is collision-free and satisfies equation \eqref{N body}, for any $t \in (0, S)$;
\item[(ii).] there is a subset of indices $\I \subset \N$ (with its cardinality $|\I| \ge2$), such that $z$ has an isolated $\I$-cluster collision at $t_0=0$ or $S$, i.e., 
\begin{align*}
z_j(t_0)=z_k(t_0), \;\; & \forall \{j \ne k \} \subset \I,\\
z_j(t_0) \ne z_k(t_0), \;\; & \forall j \in \I, \; \forall k \in \N \setminus \I;
\end{align*}
\item[(iii).] $z_j(t) \in i \rr,$ $\forall t \in [0, S]$ and $\forall j \in \N$.
\end{enumerate}
Furthermore, we define a subset of $\{0, \pm 1\}^N$ as following
\begin{equation} 
\label{eqn:cond tau} \mf{T}:=\big\{ \tau = (\tau_j)_{j \in \N}| \; \tau_j \ne \tau_k, \text{ for some } \{j \ne k\} \subset \I \text{ and } \tau_l =0, \; \forall l \in \N \setminus \I  \big\}
\end{equation}

The following three local deformation lemmas will be the key to our proof.
\begin{lm}
  \label{deform1} If $t_0=0$ and $\tau \in \mf{T}$, then for $\ep_0>0$ small enough, there exist $h \in H^1([0,T], \rr)$ and a path $z^{\ep_0} =(z^{\ep_0}_j)_{j \in \N} \in H^1([0, S], \cc^N)$, with $z^{\ep_0}_j(t) =z_j(t)+ \ep_0 h(t) \tau_j$, $t \in [0, S]$, $\forall j \in \N$,  satisfying $\ac(z^{\ep_0}, S)< \ac(z, S)$ and the following:
\begin{enumerate}
\item[(a).] $h(t)=1$, $\forall t \in [0, \dl_1]$, for some $0< \dl_1=\dl_1(\ep_0)< S$ small enough; 
\item[(b).] $h(t) =0$, $\forall t \in [\dl_2, S]$, for some $\dl_1< \dl_2=\dl_2(\ep_0) < S$ small enough;
\item[(c).] $h(t)$ is decreasing for $t \in [\dl_1, \dl_2]$. 
\end{enumerate}
 \end{lm}

 \begin{lm}
  \label{deform2} If $t_0=S$ and $\tau \in \mf{T}$, then for $\ep_0>0$ small enough, there exist $h \in H^1([0,T], \rr)$ and a path $z^{\ep_0}=(z^{\ep_0}_j)_{j \in \N} \in H^1([0, S], \cc^N)$, with $z^{\ep_0}_j(t)=z_j(t)+ \ep_0 h(t) \tau_i$, $t \in [0,S]$, $\forall j \in \N$, satisfying $\ac(z^{\ep_0}, S)< \ac(z, S)$ and the following:
\begin{enumerate}
\item[(a).] $h(t)=1$, $\forall t \in [S-\dl_1, S]$, for some $0< \dl_1=\dl_1(\ep_0)< S$ small enough; 
\item[(b).] $h(t) =0$, $\forall t \in [0, S-\dl_2]$, for some $\dl_1< \dl_2=\dl_2(\ep_0) < S$ small enough;
\item[(c).] $h(t)$ is increasing for $t \in [S-\dl_2, S-\dl_1]$. 
\end{enumerate}
 \end{lm}
 
While we can apply the above two lemmas to isolated collision occurring at the boundaries of a fundamental domain, we need another lemma for those occurring in the interior. For this, let $t_0=0$ and extend the domain of $z(t)$ to $t \in [-S, S]$. Furthermore we assume $z$ satisfies the following additional conditions: 
\begin{enumerate}
\item[(iv).] $z(t)$ is collision-free and satisfies equation \eqref{N body}, for any $t \in (-S, 0)$;
\item[(v).] $z_j(t) \in i \rr,$ $\forall t \in [-S, 0]$ and $\forall j \in \N$;
\item[(vi).] $K(\dot{z}(t))-U(z(t))= \text{Constant}$, for any $t \in (-S, S) \setminus \{0\}$.
\end{enumerate}
 
 \begin{lm}
  \label{deform3} If $\tau \in \mf{T}$, then for $\ep_0>0$ small enough, there exist $h \in H^1([-S,S], \rr)$ and a path $z^{\ep_0}=(z^{\ep_0}_j)_{j \in \N} \in H^1([-S, S], \cc^N)$, with $z^{\ep_0}_i(t)=z_i(t)+ i \ep_0 h(t) \tau_i, t \in [-S, S]$, $\forall j \in \N$, satisfying $\ac(z^{\ep_0}, -S, S)< \ac(z, -S, S)$ and the following:
\begin{enumerate}
\item[(a).] $h(t)=1$, $\forall t \in [-\dl_1, \dl_1]$, for some $0< \dl_1=\dl_1(\ep_0)< S$ small enough; 
\item[(b).] $h(t) =0$, $\forall t \in [-S, -\dl_2] \cup [\dl_2, S]$, for some $\dl_1< \dl_2=\dl_2(\ep_0) < S$ small enough;
\item[(c).] $h(t)$ is increasing for $t \in [-\dl_2, -\dl_1]$;
\item[(d).] $h(t)$ is decreasing for $t \in [\dl_1, \dl_2]$.  
\end{enumerate} 
\end{lm}

Only the proof of Lemma \ref{deform1} will be given in the following (Lemma \ref{deform2} can be proven similarly after reversing the time parameter and Lemma \ref{deform3} follows directly once we have the previous two lemmas), so let's assume $z(0)$ is an isolated $\I$-cluster collision in the following. Before we can proceed, some notations and technical lemmas need to be introduced first. 

We define the Lagrange and the corresponding action functional of the $\I$-body sub-system as following: 
$$L_{\I}(z,\zd) := \kk(\zd) + \uk(z), \;\;  \ack(z, T) := \int_0^T L_{\I}(z, \zd) \,dt,$$
$$ \kk(\zd) := \ey \sum_{j \in \I} |\zd_j|^2,\;\; \uk(z) := \sum_{\{j <k\} \in \I} \frac{1}{|z_j - z_k|}.  $$

\begin{dfn} \label{centered}
 we say $w$ is a $\I$-configuration, if $w = (w_j)_{j \in \I} \in \cc^{|\I|}$, and a centered $\I$-configuration, if $\sum_{j \in \I} w_j = 0$. 
\end{dfn}

Let $\zh(t) =\hat{x}(t)+i \hat{y}(t) = \frac{1}{|\I|} \sum_{j \in \I} z_j(t)$ be the center of mass of the $\I$-body sub-system. By condition (iii) given above, we know $x_j(t) \equiv 0$, $\forall j \in \I$, and so is $\hat{x}(t)$. As a result, we define $iv(t)= i(v_j(t))_{j \in \I}$ as
  $$ v_j(t)= y_j(t)- \hat{y}(t), \;\; \forall j \in \I.$$
Each $iv_j(t)$ represents the relative position of $m_j$, $j \in \I$, with respect to the center of mass of the $\I$-body sub-system.

\begin{dfn}
 \label{blow up} For any $0<\lmd<1$, we define $z^{\lmd} \in H^1([0, S/\lmd], \cc^N)$, $z^{\lmd}(t) = \lmd^{-\se} z(\lmd t),$ as the \textbf{$\lmd$-blow-up} of $z$, and $ v^{\lmd} \in H^1([0, S/ \lmd], \rr^{|\I|})$, $v^{\lmd}(t)= \lmd^{-\se} v(\lmd t)$, as 
 the \textbf{$\lmd$-blow-up} of $v$. 
\end{dfn}

Let $\sg(t) = \frac{v(t)}{|v(t)|}$ be the normalization of $v(t)$, where $|v(t)| = \big(\sum_{j \in \I} |v_j(t)|^2\big)^{\ey}$. It is well known that as $\{t_n\} \searrow 0$, the limit of $i\sg(t_n)$, if exists, is a central configuration of the $\I$-body problem, for a proof see \cite{FT04}. Since the space of all normalized $\I$-configuration is compact, we can always find a sequence of positive numbers $\{ \lmd_n \} \searrow 0$ with the following limit exist
$$ \lim_{n \to \infty} i\sg(\lmd_n) = i \sgt = i(\sgt_j)_{j \in \I}, \;\; \text{ with } \sgt_j \in \rr. $$
Then $i \sgt$ is a $\I$-central configuration, and $i\vt(t)$, $t \in [0, +\infty),$ given below is a homothetic-parabolic solution associated with $i\sgt$
\begin{equation}
 \label{vtilde} i\vt(t) = i(\vt_j(t))_{j \in \I}, \text{ where } \vt_j(t) = (\kp t)^{\se} \sgt_j, \;\; \forall j \in \I,
 \end{equation}
where $\kp$ is a positive constant determined by $i \tilde{\sg}$. 

The homothetic-parabolic solution $i\vt(t)$ is related to the isolated collision solution $z(t)$ through the blow-up's $v^{\lmd_n}(t)$ in the following way (for a proof see \cite[(7.4)]{FT04}). 

\begin{prop}
 For any $T>0$, the sequences $\{v^{\lmd_n} \}$ and $\{ \frac{d v^{\lmd_n}}{dt} \}$ converge to $\vt$ and its derivative $\dot{\vt}$ correspondingly. Furthermore the convergences are uniform on $[0, T]$ and on compact subsets of $(0,T]$ correspondingly. 
\end{prop}

A local deformation lemma of the above homothetic-parabolic solution will be given first. The key to its proof is the following function first introduced by Montgomery in \cite{Mo99}:
For any $T> \ep >0$, let $f \in C([0, T], \rr)$ be defined as following:
\begin{equation} \label{f}
 f(t) = \begin{cases}
         1, & \text{ if } t \in [0, \ep^{\frac{3}{2}}], \\
         1 + (\ep^{\frac{3}{2}}-t)/\ep, & \text{ if } t \in [\ep^{\frac{3}{2}}, \ep^{\frac{3}{2}}+ \ep], \\
         0, & \text{ if } t \in [\ep^{\frac{3}{2}} + \ep, T].
        \end{cases} 
\end{equation}
For any $\tau \in \mf{T}$ and $\ep>0$ small enough, define $\vt^{\ep}(t) =(\vt^{\ep}_j(t))_{j \in \I}$ as 
\begin{equation}
\label{eq: v tilde ep} \vt^{\ep}_j(t)=i\vt(t)+ \ep f(t) \tau_j, \;\; t \in [0, +\infty).
\end{equation}
\begin{lm}
 \label{deform0} For any $T>0$, $\ack(\vt^{\ep}, T) < \ack(i\vt, T)$, for $\ep>0$ small enough.
\end{lm}
\begin{proof}
 By the definition of $\vt^{\ep}$ and $f(t)$,  
 \begin{align*}
  & \ack(\vte, T) - \ack(i\vt, T) \\ &  = \int_{\ep^{\frac{3}{2}}}^{\ep^{\frac{3}{2}}+\ep} \kk(\vte) - \kk(i\vt) +\int_0^{\ep^{\frac{3}{2}}} \uk(\vte) - \uk(i\vt) + \int_{\ep^{\frac{3}{2}}}^{\ep^{\frac{3}{2}}+\ep} \uk(\vte) - \uk(i\vt)\\
  & := A_1 + A_2 + A_3.
 \end{align*}
 We estimate each $A_j$, separately in the following. 
 
 For $A_1$, notice that $|\dot{\vt}^{\ep}_j|^2 = |\ep \dot{f} \tau_j|^2 + |\dot{\vt}_j|^ 2,$ $ \forall j \in \I$, then by the definition of $f(t)$,  
  \begin{equation}
    A_1= \ey \sum_{j \in \I} \int_{\ep^{\frac{3}{2}}}^{\ep^{\frac{3}{2}}+\ep} \tau_j^2 \,dt = \frac{\ep}{2} \sum_{j \in \I} \tau_j^2 = C_1 \ep,
  \end{equation}
  $C_1$ is a positive constant, as $\tau_j \ne 0$ for some $j \in \I$.

  To estimate $A_2$, we introduce a new time parameter $s = t^{\se}/\ep,$ then 
 \begin{align*}
  A_2 & = \sum_{\{j<k\} \subset \I} \int_0^{\ep^{\frac{3}{2}}} |\ep(\tau_j - \tau_k)+ i(\vt_j(t)-\vt_k(t))|^{-1} - |\vt_j(t) - \vt_k(t)|^{-1} \,dt \\
  & = \sum_{\{j<k\} \subset \I} \int_0^{\ep^{\frac{3}{2}}} \big(\ep^2(\tau_j - \tau_k)^2+ (\kp t)^{\frac{4}{3}}(\sgt_j - \sgt_k)^2 \big)^{-\ey} - |(\kp t)^{\se}(\sgt_j - \sgt_k)|^{-1} \,dt \\
  & = \frac{3 \ep^{\ey}}{2}\sum_{\{j<k\} \subset \I}  \int_0^1 \big \{ \big((\tau_j - \tau_k)^2  +\kp^{\frac{4}{3}} s^2 (\sgt_j- \sgt_k)^2 \big)^{-\ey} -|\kp^{\se}s(\sgt_j -\sgt_k)| \big \} s^{\ey} \,ds \\
  &  = -C_2 \ep^{\ey}.
  \end{align*}
  The last equality holds for some positive constant $C_2$, as $\tau_j \ne \tau_k$ for some $j \ne k \in \I$. 
  
  For $A_3$, notice that for any $\{j< k\} \subset \I$,
  $$ U_{j,k}(\vt^{\ep}(t)) - U_{j,k}(i\vt(t)) \le 0, \;\; \forall t \in [\ep^{\frac{3}{2}}, \ep^{\frac{3}{2}}+\ep]. $$
  Then $A_3 \le 0. $ Combine the above estimates, we get for $\ep>0$ small enough, 
  $$ \ack(\vte, T) - \ack(i\vt, T) \le C_1 \ep-C_2 \ep^{\ey} <0.$$
\end{proof}
  
\begin{lm}
	\label{deform limit} For any $\tau \in \mf{T}$ and $g(t) \in H^1([0,T], \rr)$, where $g(t)$ is $C^1$ in a neighborhood of $T$. Define $\phi =(\phi_j)_{j \in \I} \in H^1([0, T], \rr^{|\I|})$ as $\phi_j(t) = g(t)\tau_j$, $\forall t \in [0, T]$, $\forall j \in \I$ , and $\{ \psi_n \in H^1([0, T], \cc )\}_{n \in \zz^+}$ as 
	\begin{equation}
	 \label{psi} \psi_n(t) = 
	\begin{cases}
	i\vt(t) - i\vt^{\lmd_n}(t) & \text{ if } t \in [0, T-\frac{1}{N_n}], \\
	N_n(T-t)\big(i\vt(t) - i\vt^{\lmd_n}(t)\big) & \text{ if } t \in  [T-\frac{1}{N_n}, T], 
	\end{cases} 
	\end{equation}	
	where $\{N_n \}$ is a sequence of positive integers goes to infinity, then  
	$$ \lim_{n \to \infty} \ac(z^{\lmd_n} + \phi + \psi_n, T) - \ac(z^{\lmd_n}, T) = \ack( \phi+i\vt, T) - \ack(i\vt, T). $$
\end{lm}
\begin{proof}	
	If $\phi(t)$ is a centered $\I$-configuration, this result was ready proved in \cite[(7.9)]{FT04}. In their proof this condition was used to show that,
	\begin{equation}
	\label{a0} \lk(iv^{\lmd} + \phi) - \lk(iv^{\lmd}) = \lk(z^{\lmd}+ \phi) - \lk(z^{\lmd}), \;\; \forall \lmd >0.
	\end{equation} 
	In the following, we show \eqref{a0} still holds even when $\phi(t)$ is not centered. It is enough to demonstrate it for $\lmd=1$. The others are the same. 
	
	Recall that $iv_j(t)= z_j(t)-\hat{z}(t)= iy_j(t)-i\hat{y}(t)$, $\forall j \in \I$. Hence
	\begin{equation}
	\label{a1} \uk(iv(t)) = \uk(z), \;\; \uk(\phi(t)+iv(t))  = \uk(\phi(t)+z(t)),
	\end{equation} 
		\begin{align} \label{a2}
	\begin{split}
		2 \kk(\zd) &  = \sum_{j \in \I}|\zd_j|^2 = \sum_{j \in \I}|iy_j|^2= \sum_{j \in \I}|i\dot{v}_j + i\dot{\hat{y}}|^2 \\
		&  = \sum_{j \in \I}|\dot{v}_j|^2 + |\I| |\dot{\hat{y}}|^2 = 2 \kk(i\dot{v})+ |\I| |\dot{\hat{y}}|^2. 
	\end{split}
	\end{align}
	Meanwhile by the fact that $\phi_j(t)  \in \rr$, for any $j \in \I$, we have
	\begin{align}
	\label{a3} \begin{split}
	2 \kk(z +\phi) & = \sum_{j \in \I} |i \dot{v}_j + i \dot{\hat{y}} + \dot{\phi}_j|^2 = \sum_{j \in \I}|i\dot{v}_j + i \hat{y}|^2 + \sum_{j \in \I} |\dot{\phi}_j|^2  \\
	& =\sum_{ j \in \I} |\dot{v}_j|^2 + \sum_{ j \in \I} |\dot{\phi}_j|^2 + |\I| |\dot{\hat{y}}|^2 = \sum_{ j \in \I} |i \dot{v}_j + \dot{\phi}_j|^2 + |\I| |\dot{\hat{y}}|^2\\
	&  = 2 \kk(i\dot{v} + \dot{\phi}) + |\I||\dot{\hat{y}}|^2. 
	\end{split}
	\end{align}
	Combining \eqref{a1}, \eqref{a2} and \eqref{a3}, we get
	$$ \lk(z+\phi) - \lk(iv+\phi) = \lk(z) - \lk(iv) = \ey |\I| |\dot{\hat{y}}|^2. $$
	This establishes \eqref{a0}. The rest of the proof is exactly the same as \cite[(7.9)]{FT04}. We will not repeat the details here. 	
	
\end{proof}

With the above results, now we can give a proof of Lemma \ref{deform1}.  
 \begin{proof}

  [\textbf{Lemma \ref{deform1}}]
  Choose a $T \in (0, S)$. Let $f(t)$ and $\{\lmd_n\}_{n \in \zz^+} \searrow 0$ be defined as above, for $\ep>0$ small, we set $\phi(t)= (\phi_j(t))_{j \in \I}$ as $\phi_j(t) = \ep f(t) \tau_j$, $\forall t \in [0, T]$, $\forall j \in \I$. By Lemma \ref{deform limit}, there is  a sequence of functions $\{ \psi_n \}_{n \in \zz^+}$ defined by \eqref{psi} satisfying
  $$ \lim_{n \to \infty} \ac(z^{\lmd_n} + \phi + \psi_n, T) - \ac(z^{\lmd_n}, T) = \ack(i\vt + \phi, T) - \ack(\vt, T). $$
  
  For each $n$, define a $\ztl^{\lmn} \in H^1([0, S/\lmn], \cc^N)$ as following 
  $$ \ztl^{\lmn}(t) = \begin{cases}
                      z^{\lmn}(t) & \text{ if } t \in [T, S/\lmn], \\
                      z^{\lmn}(t) + \phi(t) + \psi(t) & \text{ if } t \in [0, T].
                     \end{cases} $$
  By the definition of $\ztl^{\lmn}$ and Lemma \ref{deform limit},
  $$\lim_{n \to \infty} \ac(\ztl^{\lmn}, S/\lmn) - \ac(z^{\lmn}, S/\lmn)= \ack(i\vt + \phi, T) - \ack(i\vt, T). $$
  For any $j \in \I$, by \eqref{eq: v tilde ep},  $\vt^{\ep}_j(t)= i\vt_j(t)+\ep f(t) \tau_j= i\vt_j(t)+ \phi_j(t)$. By Lemma \ref{deform0}
  $$ \lim_{n \to \infty} \ac(\ztl^{\lmn}, S/\lmn) - \ac(z^{\lmn}, S/\lmn)= \ack(\vt^{\ep}, T) - \ack(i\vt, T) <0, $$
  so for $n$ large enough
  \begin{equation}
  \label{eq: zt lmn} \ac(\ztl^{\lmn}, S/\lmn) < \ac(z^{\lmn}, \lmn/S). 
  \end{equation}
   
  Now define $\ztl_n \in H^1([0, S], \cc^{N})$ as $\ztl_n(t) = \lmn^{\se} \ztl^{\lmn}(t/\lmn)$. By \eqref{eq: zt lmn}, a straight forward computation shows $\ac(\ztl_n, S) < \ac(z, S)$, for $n$ large enough. Notice that for any $t \in [\lmn T, S]$, $\ztl_n(t) = z(t),$ and for any $t \in [0, \lmn T]$, 
  \begin{align*}
   \ztl_n(t) & = z(t) + \lmn^{\se}\big(\phi(t/\lmn) + \psi (t/\lmn) \big) = z(t) + \lmn^{\se} \ep f(t/\lmn) \tau + \lmn^{\se} \psi(t/\lmn). 
  \end{align*}
  By the definition of $f$ and $\psi$, for $n$ large enough, $\ztl_n$ satisfies all the conditions that are required for $\zel$. This finishes our proof.
 \end{proof}

 \begin{rem}
  Although we only considered the planar problem here. For the general $N$-body problem in $\rr^d$ with $d \ge 2$, if $z(t)$ is solution with an isolated collision, satisfying $z_j(t) \in V$, $\forall t$ and $\forall j \in \N$, where $V$ is a $d-1$ dimensional linear subspace of $\rr^d$. Then local deformation as lemmas \ref{deform1}, \ref{deform2} and \ref{deform3} can be proven similarly, when we only try to deform the path $z(t)$ along the directions orthogonal to $V$.  
 \end{rem}

Now we are ready to prove Lemma \ref{degenerate}.
 
 \begin{proof}
  
  [\textbf{Lemma \ref{degenerate}}] For simplicity, let $z = \zo$. By a contradiction argument, let's assume $z_0(N/2) - z_0(0)=0$. By \eqref{0}, $z_0(t) \equiv 0$, $ \forall t \in \rr$. This means every body stays on the imaginary axis all the time. 

  As a result, $z$ is a minimizer of the action functional $\ac$ among all loops in $\Lmd_N^{D_N} \cap H^1(\rr/N\zz, (i\rr)^N)$ that satisfies the $\om$-topological constraints, which means $z(t)$ satisfies equation \eqref{N body}, whenever it is collision-free. Furthermore since $\ac(z)$ is finite, it only contains isolated collision (see \cite[Section 3.3]{C02} or \cite[Section 5]{FT04}). 
    
  First let's assume $N=2n$. Then $z_0(0) = z_n(0)$, and $z$ has an isolated $\I$-cluster collision, with $\{0, n\} \subset \I$, at $t=0$. Choose a $\tau \in \mf{T}$ with
  \begin{equation}
   \label{eq: tau 1}  \tau_0 = -1, \; \tau_n = 1 \text{ and } \tau_j =0, \;\; \forall j \in \N \setminus \{0, n\}.
  \end{equation} 
  By Lemma \ref{deform1}, there is a $\zel \in H^1([0, 1/2], \cc^N)$, which is a local deformation of $z$, satisfying $\ac(\zel, 1/2) < \ac(z, 1/2)$, for $\ep_0>0$ small enough. Since only the paths of $m_0$ and $m_n$ were deformed, by the properties listed in Lemma \ref{deform1}, $\zel$ satisfies the monotone constraints and as a loop is contained in $\lo$. This gives us a contradiction. 

  For the rest, we assume $N=2n+1$. The proof is more complicate in this case, because although there will always be an isolated collision, which will be shown in a moment, the collision time and masses that are involved in the collision may vary.

  Without loss of generality, let's assume $\om_1=1$. Then 
  \begin{equation}
  \label{eq: y0 0 +}  y_0(0) =0, \;\; y_0(1/2) \ge 0. 
  \end{equation}
  Meanwhile if $\om_j = 1$, $\forall j \in \N \setminus \{0, 1\}$, then
  \begin{equation}
  \label{eq: yn + 0}  y_n(0) \ge 0, \;\; y_n(1/2) =0. 
  \end{equation}
  On the other hand, if there is a $ 1 \le j_0 \le N-2$, such that $w_j=1,\; \forall j \in \{1, \dots, j_0\}$ and $w_{j_0+1}=-1,$ then by \eqref{eq: z0 zj zN-j}, 
  \begin{equation}
  \label{eq: y j0 + -} \begin{cases}
  y_{j_0/2}(0) \ge 0, \; y_{j_0/2}(1/2) \le 0, \;\; & \text{ if } j_0 \text{ is even}, \\
  y_{N-\frac{j_0+1}{2}}(0) \ge 0, \; y_{N-\frac{j_0+1}{2}}(1/2) \le 0, \;\; & \text{ if } j_0 \text{ is odd}. 
  \end{cases}
  \end{equation}
  As $x_j(t) \equiv 0$, $\forall j \in \N$, \eqref{eq: y0 0 +}, \eqref{eq: yn + 0} and \eqref{eq: y j0 + -} show that there must be a $k \in \N \setminus \{0\}$ and $t_0 \in [0, 1/2]$, such that $z_0(t_0) =z_k(t_0)$. In other words, there must be an isolated $\I$-cluster collision at $t=t_0$, with $\{0, k \} \subset \I$. 

  Like above we will use the local deformation lemmas given at the beginning of this section to find a new path $\zel \in H^1([0, 1/2], \cc^N)$ with $\ac(\zel, 1/2) < \ac(z, 1/2)$, for $\ep_0>0$ small enough. However $\zel$ may not satisfies the monotone constraints, which prevents it being a member of $\lo$. Fortunately after some proper modification, if needed, we can always get a $\ztl \in \lo$ with $\ac(\ztl, 1/2) \le \ac(\zel, 1/2) < \ac(z, 1/2)$, which gives us a contradiction. 

  In the following, we will show how to find such a $\ztl$. Depending on the value of $k$ and $t_0$, six different cases need to be considered. During the proof, it will be helpful to keep \eqref{eq: monotone 1} and \eqref{eq: monotone 3} in mind.

  \begin{figure}
  	\centering
  	\includegraphics[scale=0.7]{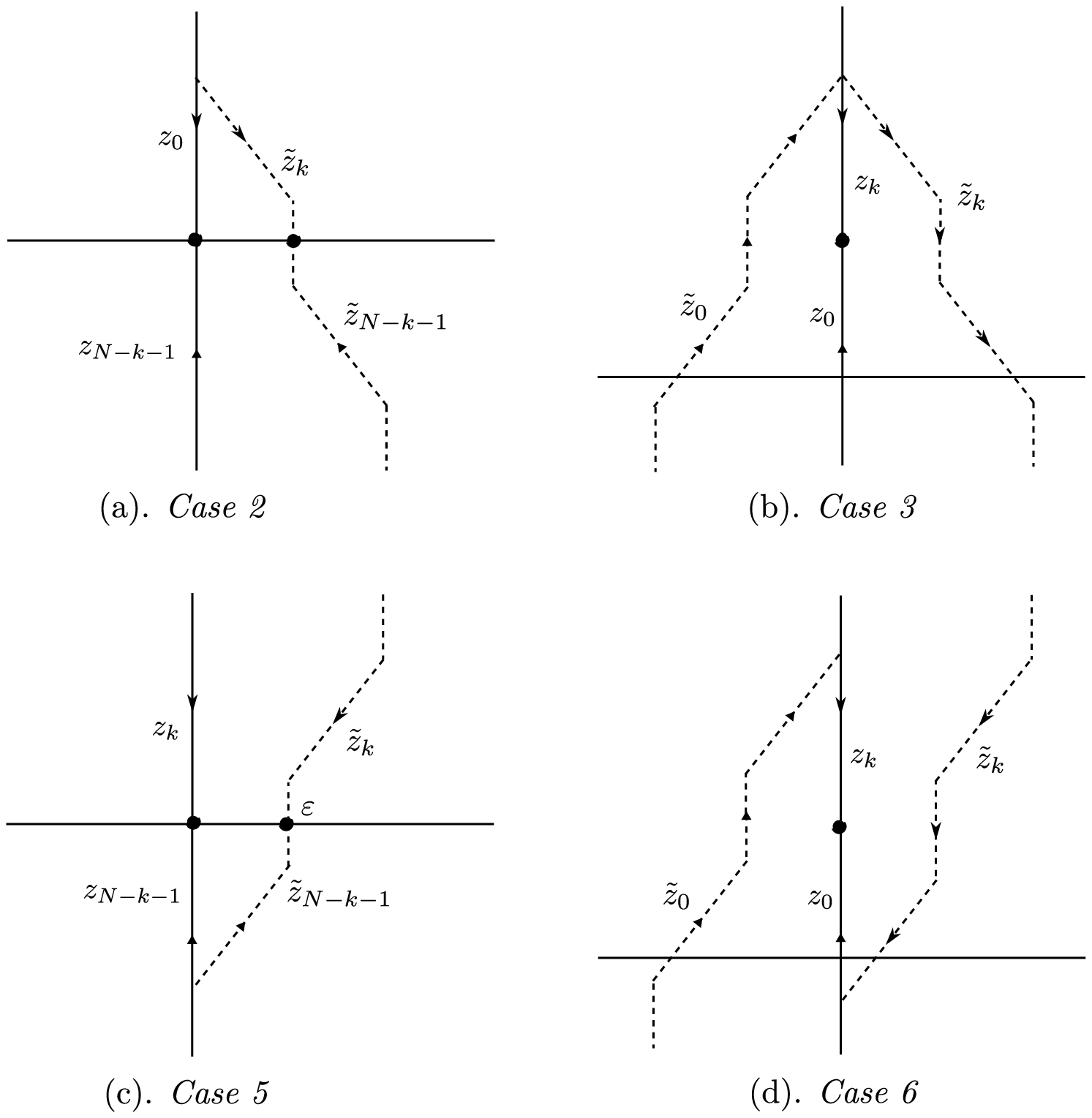}
  	\caption{}
  	\label{fig:C5}
  \end{figure}

  \emph{Case 1}: $ k \in \{1, \dots, n\}$ and $t_0=0$. By \eqref{eq: symm 3}, $z_0(0)=z_k(0)= z_{N-k}(0)$, so $\{0, k, N-k \} \subset \I$. Choose a $\tau \in \mf{T}$ with 
  $$ \tau_0 = -1 \; \text{ and } \; \tau_j = 0, \; \forall j \in \N \setminus \{0\}.$$
  By Lemma \ref{deform1}, for $\ep_0>0$ small enough, there is a $\zel \in H^1([0, 1/2], \cc^N)$, which is a local deformation of $z$, satisfying $\ac(\zel, 1/2) < \ac(z, 1/2)$. Since only the path of $m_0$ is deformed, with the properties listed in Lemma \ref{deform1}, $\zel \in \lo$. We set $\ztl= \zel$. 

  \emph{Case 2}:  $ k \in \{1, \dots, n\}$ and $t_0=1/2$. By \eqref{eq: symm 3}, $z_0(1/2) = z_k(1/2) = z_{N-k-1}(1/2)$, so $\{0, k, N-k-1\} \subset \I$. Choose a $\tau \in \mf{T}$ with 
   $$ \tau_{k}= \tau_{N-k-1} =1 \; \text{ and } \;  \tau_j = 0,  \forall j \in \I \setminus \{k, N -k-1 \}. $$
 By Lemma \ref{deform2}, there is a $\zel \in \in H^1([0, 1/2], \cc^N)$, which is a local deformation of $z$, satisfying $\ac(\zel, 1/2) < \ac(z, 1/2)$. Here only the paths of $m_k$ and $m_{N-k-1}$ are deformed. However $\zel_{N-k-1}(0)< \zel_{N-k-1}(1/2)$ violates \eqref{eq: monotone 1}, so we defined a new path $\ztl(t)$, $t \in [0, 1/2]$ as following:  
  \begin{equation*}
  \ztl(t) = \begin{cases}
       2x^{\ep_0}_j(1/2)-x_j^{\ep_0}(t)+iy_j^{\ep_0}(t), \; & \text{ if } j = N-k-1, \\
       \zel_j(t)+2x_{N-k-1}^{\ep}(1/2), \; & \text{ if } j \in \{k+1, \dots, N-k-2\}, \\
       \zel_j(t),   \; & \text{ if } j \in \N \setminus \{k+1, \dots, N-k-1\}. 
  \end{cases}
  \end{equation*}
  See Figure \ref{fig:C5}(a) for an illuminating picture. By the above definition, $\ac(\ztl, 1/2) \le \ac(\zel, 1/2)$. Furthermore it satisfies the monotone constraints and as a loop is contained in $\lo$.   

  Case $3$:  $ k \in \{1, \dots, n\}$ and $t_0 \in (0, 1/2)$. Then $z_0(t_0)= z_k(t_0)$, so $\{0, k \} \subset \I$. Choose a $\tau \in \mf{T}$ with 
  $$ \tau_0=-1, \; \tau_{k} = 1,  \text{ and } \tau_j =0, \; \forall j \in \N \setminus \{0, k \}. $$
  By Lemma \ref{deform3}, there is a $\zel \in H^1([0, 1/2], \cc^{N})$,  which is a local deformation of $z$, satisfying $\ac(\zel, 1/2) < \ac(z, 1/2)$. Only the paths of $m_0$ and $m_k$ are deformed. However $\zel_0(t_0) < \zel_0(0)$, $\zel_k(1/2) < \zel_k(t_0)$ violate the monotone constraints, so we define a new path $\ztl \in \lo$ as following: 
  $$ \ztl_0(t)= \begin{cases}
  2x^{\ep_0}_0(t_0)-x^{\ep_0}_0(t)+iy^{\ep_0}_0(t), \;\; &\forall t \in [0, t_0], \\
  \zel_0(t), \;\; & \forall t \in [t_0, 1/2],
  \end{cases} $$
  $$ \ztl_{k}(t) = \begin{cases}
  \zel_k(t), \;\; & \forall t \in [0, t_0], \\
  2x^{\ep_0}_k(t_0)-x^{\ep_0}_k(t)+iy^{\ep_0}_k(t), \;\; & \forall t \in [t_0, 1/2],
  \end{cases} $$
  and for any $t \in [0, 1/2]$, 
  $$ \ztl_j(t) = \begin{cases}
  \zel_j(t)+2x^{\ep_0}_k(t_0), \;\; & \text{ if } j \in \{k+1, \dots, N-k-1\}, \\
  z^{\ep_0}_j(t), \;\; & \text{ if } j \in \N \setminus \{0, k, \dots, N-k-1 \}. 
  \end{cases} $$
  See Figure \ref{fig:C5}(b). The rest is the same as in \emph{Case 2}. 
  
  \emph{Case 4}: $k \in \{n+1, \dots, N-1 \}$ and $t_0=0$. By \eqref{eq: symm 3}, $z_0(0)=z_k(0)= z_{N-k}(0)$, so $\{0, k, N-k \} \subset \I$. The rest follows from the same $\tau$ and argument used in \emph{Case 1}. 
  
  \emph{Case 5}: $k \in \{n+1, \dots, N-1 \}$ and $t_0=1/2$. The rest follows from the same $\tau$ and argument used in \emph{Case 2}, when $\ztl(t)$, $t \in [0, 1/2],$ is defined as following (see Figure \ref{fig:C5}(c)):   
  $$ 
  \ztl_j(t)= \begin{cases}
  2x^{\ep_0}_j(1/2)-x^{\ep_0}_j(t)+iy^{\ep_0}_j(t), \;\; & \text{ if } j =k, \\
  \zel_j(t) + 2\zel_k(1/2), \;\; & \text{ if } j \in \{N-k, \dots, k-1 \}, \\
  \zel_j(t), \;\; & \text{ if } j \in \N \setminus \{N-k, \dots, k \}.
  \end{cases}
  $$
     
  \emph{Case 6}: $k \in \{n+1, \dots, N-1 \}$ and $t_0 \in (0, 1/2)$. The rest follows from the same $\tau$ and argument used in \emph{Case 3}, when $\ztl(t)$ should is as following (see Figure \ref{fig:C5}(d)):
 
  $$ \ztl_0(t)= \begin{cases}
  2 x^{\ep_0}_0(t_0)-x^{\ep_0}_0(t)+iy^{\ep_0}_0(t), \;\; & \forall t \in [0, t_0], \\
  \zel_0(t), \;\; & \forall t \in [t_0, 1/2], 
  \end{cases} $$

  $$ \ztl_k(t) = \begin{cases}
  2x^{\ep_0}_k(t_0)-x^{\ep_0}_k(t)+i y^{\ep_0}_k(t), \;\; & \forall t \in [0, t_0], \\
  \zel_k(t), \;\; & \forall t \in [t_0, 1/2], 
  \end{cases} $$
  and for any $t \in [0, 1/2]$, 
  $$ \ztl_j(t) = \begin{cases}
  \zel_j(t)+ 2 x^{\ep_0}_k(t_0), \;\; & \text{ if } j \in \{N-k-1, \dots, k-1 \}, \\
  \zel_j(t), \;\; & \text{ if } j \in \N \setminus \{0, N-k-1, \dots, k \}. 
  \end{cases} $$
  
  This finishes our proof for $N=2n+1$, as well as the entire lemma.

 \end{proof}

\mbox{}

\emph{Acknowledgments}. The author wishes to express his gratitude to Kuo-Chang Chen, Alain Chenciner, Jacques F\'ejoz, Rick Moeckel, Richard Montgomery, Carles Sim\'o for their interests and comments on this work. He thanks Ke Zhang for his continued support and encouragement. Part of the work was done when the author was a visitor at Shandong University. He thanks the hospitality and support of Xijun Hu through NSFC(No. 11425105).

\bibliographystyle{abbrv}
\bibliography{ref-choreography}

\end{document}